\documentclass[preprint,10pt]{elsarticle}

\usepackage{amsmath,amssymb,amsfonts,cmmib57,mathdots,xcolor}
\usepackage[notcite,notref]{showkeys}
\usepackage{amsfonts,amssymb,amsmath,amsthm,cmmib57,color,graphics,mathdots}
\usepackage[latin1]{inputenc}
\usepackage{enumerate}

\newcommand{\FF}{{\mathbb F}}

\newtheorem{theorem}{Theorem}[section]

\newtheorem{definition}[theorem]{Definition}
\newtheorem{corollary}[theorem]{Corollary}
\newtheorem{lemma}[theorem]{Lemma}
\newtheorem{remark}[theorem]{Remark}
\newtheorem{example}[theorem]{Example}
\newtheorem{prop}[theorem]{Proposition}

\DeclareMathOperator{\diag}{diag}
\DeclareMathOperator{\su}{su}
\DeclareMathOperator{\rev}{rev}

\journal{Linear Algebra and its Applications}

\begin{document}

\begin{frontmatter}

\title{Symmetric and skew-symmetric  block-Kronecker linearizations \tnoteref{proj}}

\tnotetext[proj]{The second author has been supported by Engineering and Physical Sciences Research Council grant EP/I005293, and the third author has been supported by Alexander von Humboldt Foundation.}

\author[label1]{Heike Fa\ss bender}
\address[label1]{Institut \emph{Computational Mathematics}, AG Numerik, Technische Universit\"at Braunschweig, 38092 Braunschweig, Germany}
\ead{h.fassbender@tu-braunschweig.de}

\author[label2]{Javier P\'erez\corref{cor1}}
\address[label2]{School of Mathematics, The University of Manchester, Oxford Road, M13 9PL, Manchester, UK}
\ead{javierpa@gmail.com}
\cortext[cor1]{Corresponding author}

\author[label1]{Nikta Shayanfar}
\ead{n.shayanfar@tu-braunschweig.de}

\begin{abstract}
Many applications give rise to structured matrix polynomials.
The problem of constructing structure-preserving strong linearizations of structured matrix polynomials is  revisited in this work and in the forthcoming ones \cite{PartII,PartIII}.
With the purpose of providing a much simpler framework for structure-preserving linearizations for symmetric and skew-symmetric matrix polynomial than the one based on  Fiedler pencils with repetition, we introduce in this work the families of (modified) symmetric and skew-symmetric block Kronecker pencils.
These families provide a  large arena of structure-preserving strong linearizations of symmetric and skew-symmetric matrix polynomials.
When the matrix polynomial has degree odd, these linearizations are  strong regardless of whether the matrix polynomial is regular or singular, and  many of them give rise to structure-preserving companion forms.
When some generic nonsingularity conditions are satisfied, they are also strong linearizations for even-degree regular matrix polynomials.
Many examples of structure-preserving linearizations obtained from Fiedler pencils with repetitions found in the literature are shown to belong (modulo permutations) to these families of linearizations.
In particular, this is shown to be true for the well-known block-tridiagonal symmetric and skew-symmetric companion forms.
Since the families of symmetric and skew-symmetric block Kronecker pencils belong to the  recently introduced set of minimal bases pencils \cite{Fiedler-like}, they inherit all its desirable properties for numerical applications.
In particular, it is shown that eigenvectors, minimal indices, and minimal bases
of matrix  polynomials are easily recovered from those of any of the linearizations constructed in this work.
\end{abstract}


\end{frontmatter}

\section{Introduction}\label{sec:intro}
We consider in this work $n\times n$ matrix polynomials over an arbitrary field $\mathbb{F}$ of the form
\begin{equation}\label{eq:poly}
P(\lambda) = \sum_{k=0}^{d}P_k\lambda^k, \quad \mbox{with} \quad P_0,P_1,\hdots,P_d\in\mathbb{F}^{n\times n},
\end{equation}
with some algebraic structure.
In applications, the most relevant of these structures are:

\begin{itemize}
\item symmetric: $P_i=P_i^T$; and skew-symmetric: $P_i=-P_i^T$;

\item palindromic: $P_{d-i}=P_i^T$; and   anti-palindromic: $P_{d-i}=-P_i^T$;

\item alternating: $P(-\lambda) = P(\lambda)^T $ or $P(-\lambda)=-P(\lambda)^T$;
\end{itemize}

\noindent together with their variants involving conjugate-transposition instead of transposition when $\mathbb{F}=\mathbb{C}$; see, for example, \cite{GoodVibrations,Alternating,Palindromic,Skew} and the references therein.
 Since the structure of a matrix polynomial is reflected in its spectrum, it is of fundamental importance 
to exploit its structure in numerical methods to solve polynomial eigenvalue problems \cite{GoodVibrations}. 
Otherwise, it is well known that numerical methods that ignore this may produce results which are meaningless in physical applications \cite{QEP}.

The most common approach to solve the polynomial eigenvalue problem associated with a matrix polynomial $P(\lambda)$ is to linearize $P(\lambda)$ into a matrix pencil (i.e., matrix polynomial of degree 1).
Linearization transforms the polynomial eigenvalue problem into an equivalent generalized eigenvalue problem, which can be solved using standard techniques such as the QZ algorithm \cite{PEPbook}.
One of the preferred approaches to develop structured numerical methods for computing the eigenvalues of structured matrix polynomials is devising structured linearizations \cite{GoodVibrations}.
In practice, the most frequently used linearization to solve a polynomial eigenvalue problem is the Frobenius companion form \cite{Fiedler,Frobenius}.
However, Frobenius companion forms do no share in general the structure that $P(\lambda)$ might posses. 
Hence, finding linearizations that retain whatever structure the matrix polynomial might possess is a fundamental problem in the theory of linearizations 
\cite{Greeks2,FPR1,FiedlerHermitian,FPR2,FPR3,VectorSpaces,
PalindromicFiedler,symmetric,ChebyshevPencils,GoodVibrations,ChebyFiedler,NNT}.

There are two main sources of structure-preserving linearizations in the literature. 
The first source is based on pencils belonging to the vector space $\mathbb{DL}(P)$, introduced in \cite{DL} and further analyzed in \cite{BackErrors,symmetric,Conditioning,GoodVibrations,NNT}. 
The pencils in this vector space are easily constructible from the matrix coefficients of $P(\lambda)$, and most of them are strong linearizations when $P(\lambda)$ is regular.
However, none of these pencils is a strong linearization when $P(\lambda)$ is singular \cite{singular}. 
The second source is based on Fiedler pencils \cite{Greeks,Fiedler} and its different generalizations \cite{FPR1,FPR2,FPR3,VectorSpaces,PalindromicFiedler}. 
Using the Fielder pencils approach, structure-preserving strong  linearizations for matrix polynomials of odd degree (regardless of whether the matrix polynomial is regular or singular) have been constructed, as well as structure-preserving strong  linearizations for regular even-degree matrix polynomials under some nonsingularity conditions (in particular, nonsingular leading and/or trailing coefficients).
Fiedler pencils are easy to construct from the coefficients of the matrix polynomial, they are always strong linearizations regardless the matrix polynomial is regular or singular, and the eigenvectors, the minimal indices, and the minimal bases of any Fiedler pencil and those of the matrix polynomial are related in simple ways.
However, proving all these results requires considerable effort (for the different generalizations of Fiedler pencils the proofs become much more involved).

To overcome most of the difficulties and drawbacks with the $\mathbb{DL}(P)$ and Fiedler approaches, we present a new approach to the problem of constructing structure-preserving linearizations.
To this aim, we will first identify two subfamilies of the recently introduced family of \emph{(strong) block minimal bases pencils} \cite{Fiedler-like}, namely, \emph{structurable block Kronecker pencils} and \emph{modified structurable block Kronecker pencils}, which, as we will show in a series of three papers \cite{PartII,PartIII}, will provide a  fertile source of structure-preserving strong linearizations of structured matrix polynomials\footnote{This approach has been outlined in the recent reference \cite{Leo2016} for matrix polynomials with odd degree, where the authors construct one structure-preserving linearization for each of the structure classes mentioned at the beginning of this introduction.}.
However, the main goal of this and the forthcoming papers \cite{PartII,PartIII} is not only to provide new structure-preserving strong linearizations (nowadays plenty of them can be found in the literature), but to provide a much simpler framework for structure-preserving linearizations with the very important properties:
\begin{itemize}
\item[\rm (i)] they are easily constructable from the coefficients of the matrix polynomials;
\item[\rm (ii)]  eigenvector of regular matrix polynomials are easily recovered from those of the linearizations;
\item[\rm (iii)]  minimal bases of singular matrix polynomials are easily recovered from those of the linearizations;
\item[\rm (iv)]  there exists a simple relation between the minimal indices of singular matrix polynomials and the minimal indices of the linearizations, and such relation is robust under perturbations;
\item[\rm (v)] guarantee global backward stability of  polynomial eigenvalue problems solved via linearizations;
\item[\rm (vi)] they present one-sided factorizations (as those used in \cite{Framework}), which are useful for performing residual local (i.e., for each particular computed eigenpair) backward error and eigenvalue conditioning  analyses of regular polynomial eigenvaule problems solved by linearizations \cite{BackErrors,Conditioning}.
\end{itemize}
All these properties  are inherited from the properties of the family of (strong) minimal bases pencils, which have been proven in a clean, simple, and general way in \cite{Fiedler-like}.
Additionally, we expect this set of structure-preserving linearizations to include (maybe modulo row/column permutations and sign changes) most of the structure-preserving linearizations based on Fiedler pencils \cite{Greeks,Greeks2,FPR1,FPR2,FPR3,VectorSpaces,Curlett,PalindromicFiedler,Greeks3},  so these works provide a simplifying approach to Fiedler pencils theory.

The focus of this work is on symmetric and skew-symmetric linearizations of, respectively, symmetric and skew-symmetric matrix polynomials.
In words, given a symmetric or skew-symmetric matrix polynomial, we introduce a new approach to construct  pencils
\[
\mathcal{L}(\lambda)=\lambda \mathcal{L}_1+\mathcal{L}_0 \quad \mbox{with} \quad \mathcal{L}_1^T=\sigma \mathcal{L}_1\,\, \mbox{and}\,\, \mathcal{L}_0^T=\sigma \mathcal{L}_0,
\] 
that are strong linearizations  of $P(\lambda)$, where $\sigma=1$ corresponds to the symmetric case and $\sigma=-1$  to the skew-symmetric case. 
These strong linearizations will be obtained from our first examples of (modified) structurable block Kronecker pencils, namely, \emph{(modified) symmetrizable and skew-symmetrizable block Kronecker pencils}.
We will also show that many examples of symmetric and skew-symmetric linearizations in the literature are (modulo permutations) included in these sets of block Kronecker pencils.
In particular, this is shown to be true for the famous block-tridiagonal symmetric and skew-symmetric companion forms in \cite{Greeks,Greeks2,Skew}.
The palindromic and anti-palindromic cases will be considered in the second part  of this series of papers \cite{PartII}, while the alternating cases will be considered in the third part \cite{PartIII}.

The rest of the paper is organized as follows.
In Section \ref{sec:basic} we set the notation and review the basic definitions and required results used in the paper.
In  Section \ref{sec:pencils}, we revisit the recently introduced family of dual minimal bases pencils.
Then, in Sections \ref{sec:Kronecker_pencils_odd} and \ref{sec:Kronecker_pencils_even} we identify two subsets of the dual minimal bases pencil family, namely, symmetric block Kronecker pencils and skew-symmetric block Kronecker pencils.
These two sets will be used to construct (i) structure-preserving linearizations of symmetric or skew-symmetric odd-degree matrix polynomials that are strong regardless of whether the matrix polynomial is regular or singular, and (ii) structure-preserving linearizations for regular even-degree symmetric or skew-symmetric matrix polynomials  provided that their leading or trailing coefficients are nonsingular.
In Section \ref{sec:Hermitian} we show how to adapt our techniques to construct structure-preserving linearizations for Hermitian or skew-Hermitian matrix polynomials.
In Section \ref{sec:recovery}, we show how to recover eigenvectors, minimal bases, and minimal indices of a matrix polynomial from those of any of its linearizations obtained from (modified) symmetric or skew-symmetric block Kronecker pencils.
Finally, conclusions and future work are presented in Section \ref{sec:conclusions}.

\section{Auxiliary results, definitions, and notation}\label{sec:basic}

Throughout the paper we use the following notation.
We denote by $I_\ell$ the $\ell\times \ell$ identity matrix, and by $0$ we denote the zero matrix, whose size should be clear from the context.

Consider a $np\times np$ matrix $A$ partitioned into $p\times p$ blocks of size $n\times n$, and  denote by $A_{ij}$ the $(i,j)$-block-entries of $A$.
Then, we define the \emph{sum of the block entries} of $A$, denoted by $\su(A)$, by
\[
\su(A):=\sum_{i,j=1}^p A_{ij},
\] 
which is an $n\times n$ matrix.
Additionally, given another $np\times np$ matrix $B$ partitioned into blocks $B_{ij}$ conformable with those of $A$, we define the \emph{block Hadamard product}, denoted by $A \odot B$, by
\[
\begin{bmatrix}
A_{11} & \cdots & A_{1p} \\
\vdots & \ddots & \vdots \\
A_{p1} & \cdots & A_{pp}
\end{bmatrix}
\odot
\begin{bmatrix}
B_{11} & \cdots & B_{1p} \\
\vdots & \ddots & \vdots \\
B_{p1} & \cdots & B_{pp}
\end{bmatrix} :=
\begin{bmatrix}
A_{11}B_{11} & \cdots & A_{1p}B_{1p} \\
\vdots & \ddots & \vdots \\
A_{p1}B_{p1} & \cdots & A_{pp}B_{pp}
\end{bmatrix}.
\]
Notice that for $n=1$ the block Hadamard product reduces to the standard Hadamard product of two matrices.

Given an arbitrary field $\mathbb{F}$, we denote by $\FF[\lambda]$ and $\FF(\lambda)$, respectively, the ring of polynomials and the field of rational functions with coefficients in $\FF$ in the variable $\lambda$.
The set of $m\times n$ matrices with entries in $\FF[\lambda]$ (resp. $\FF(\lambda)$) is denoted by $\mathbb{F}[\lambda]^{m\times n}$ (resp.  $\FF(\lambda)^{m\times n}$).
The algebraic closure of $\FF$ is denoted by $\overline{\FF}$.

A matrix $P(\lambda)\in \mathbb{F}[\lambda]^{m\times n}$ is called an \emph{$m\times n$ matrix polynomial}.
If $n=1$ we also refer to $P(\lambda)$ as a \emph{vector polynomial}.
 The matrix polynomial $P(\lambda)$ in \eqref{eq:poly} is said to have \emph{degree} d if $P_d\neq 0$ and $P_{d+j}=0$, for $j>0$.
The degree of a matrix polynomial $P(\lambda)$ is denoted by $\deg(P(\lambda))$.
A number $g\geq \deg(P(\lambda))$ is  called the grade of $P(\lambda)$ if $P(\lambda)$ is expressed as $P(\lambda)=\sum_{k=0}^g P_k\lambda_k$, with $P_{d+1},\hdots,P_g=0$.
Throughout this paper when the grade of $P(\lambda)$ is not explicitly stated, we consider its grade equal to its degree.
A matrix polynomial of degree 1 is called a \emph{matrix pencil}.
For $k\geq \deg(P(\lambda))$, the \emph{$k$-reversal  matrix polynomial} of $P(\lambda)$ is 
\[
\rev_k P(\lambda) := \lambda^k P(\lambda^{-1}).
\]  
When $k=\deg(P(\lambda))$, we sometimes write $\rev P(\lambda)$ instead of $\rev_d P(\lambda)$.
In Lemma \ref{lemma:rev} we state a very simple relation between eigenvalues and eigenvectors of a matrix polynomial and its reversal, where we regard 0 and
$\infty$ as reciprocals.
\begin{lemma}\label{lemma:rev}
Let $P(\lambda)=\lambda^k P_k\in\FF[\lambda]^{n\times n}$ be a regular matrix polynomial. 
Then  $(\lambda_0,v)$ is an eigenpair of $P(\lambda)$ if and only if $(1/\lambda_0,v)$ is an eigenpair of $\rev_d P(\lambda)$.
\end{lemma}

The concept of block-transposition and block-symmetric matrix polynomials play a role in some of the developments in this work (for properties of the block-transposition operation see  \cite[Chapter 3]{MackeyThesis}).
\begin{definition}\label{def:block_transpose}
Let $H(\lambda)=[H_{ij}(\lambda)]$ be a block $r\times s$ matrix polynomial with $m\times n$ blocks $H_{ij}(\lambda)$.
The \emph{block-transpose} of $H(\lambda)$ is the $s\times r$ block matrix polynomial $H(\lambda)^{\mathcal{B}}$ with $m\times n$ blocks defined by $(H(\lambda)^{\mathcal{B}})_{ij}=H_{ji}(\lambda)$.
Additionally, we say that the matrix polynomial $H(\lambda)$ is \emph{block-symmetric} if $H(\lambda)^{\mathcal{B}}=H(\lambda)$.
\end{definition}

A matrix polynomial $P(\lambda)$ is said to be \emph{regular} if $P(\lambda)$ is square ($m = n$) and $\det P(\lambda)$ is not the identically zero polynomial. 
Otherwise, the matrix polynomial $P(\lambda)$ is said to be \emph{singular} (note that this includes all
rectangular matrix polynomials $m\neq n$).
For the eigenstructure of a matrix polynomial we will follow the same notation and definitions as in \cite[Definition 2.17]{IndexSum}.
We recall that the complete eigenstructure of a matrix polynomial is its finite and infinite elementary divisors, together with its left and right minimal indices.

When a matrix polynomial $P(\lambda)\in\FF[\lambda]^{m\times n}$ is singular, there may exist vectors polynomials $x(\lambda)\in\mathbb{F}(\lambda)^{n\times 1}$ and $y(\lambda)^T\in\mathbb{F}(\lambda)^{1\times m}$ such that $P(\lambda)x(\lambda)=y(\lambda)^TP(\lambda)=0$.
This motivates the following definition.
\begin{definition}
The \emph{left and right nullspaces} of a singular matrix polynomial $P(\lambda)\in\FF[\lambda]^{m\times n}$, denoted by $\mathcal{N}_l(P)$ and $\mathcal{N}_r(P)$, respectively, are the vector spaces
\begin{align*}
\mathcal{N}_l(P):=&\{y(\lambda)^T\in\mathbb{F}(\lambda)^{1\times m} \quad \mbox{such that} \quad y(\lambda)^TP(\lambda) = 0\},\\
\mathcal{N}_r(P):=&\{x(\lambda)\in\mathbb{F}(\lambda)^{n \times 1} \quad \mbox{such that} \quad P(\lambda)x(\lambda) = 0\}.
\end{align*} 
\end{definition}

It is not difficult to show that it is always possible to find bases for $\mathcal{N}_l(P)$ and $\mathcal{N}_r(P)$ consisting entirely of vector polynomials.
The \emph{order} of a vector polynomial basis is defined as the sum of the  degrees of its  vector polynomials  \cite[Definition 2]{Forney}.
Among all the possible polynomial bases of $\mathcal{N}_l(P)$ and $\mathcal{N}_r(P)$, we are interested in the ones with least order.
\begin{definition}{\rm \cite[Definition 3]{Forney}}
Let $\mathcal{V}$ be a subspace of $\mathbb{F}(\lambda)^{n\times 1}$.
A \emph{minimal basis} of $\mathcal{V}$ is any polynomial basis of $\mathcal{V}$ with least order among all polynomial bases. 
\end{definition}

Minimal bases of $\mathcal{N}_l(P)$ and $\mathcal{N}_r(P)$ are not unique, but the order list  of the degrees of the vector polynomials in any minimal basis of $\mathcal{N}_l(P)$ and $\mathcal{N}_r(P)$ is always the same.
This motivates the following definition (see \cite[Definition 4]{Forney}).
\begin{definition}
Let $P(\lambda)$ be an $m\times n$ singular matrix polynomial, and let $\{y_1(\lambda)^T,\hdots,y_q(\lambda)^T\}$ and $\{x_1(\lambda),\hdots,x_p(\lambda)\}$ be minimal bases of $\mathcal{N}_l(P)$ and $\mathcal{N}_r(P)$, respectively, ordered such that $\deg (y_1(\lambda))\leq \cdots \leq \deg (y_q(\lambda))$ and \break $\deg (x_1(\lambda))\leq  \cdots \leq \deg (x_p(\lambda))$.
Let $\mu_j=\deg (y_j(\lambda))$, for $j=1,2,\hdots,q$, and  $\epsilon_j = \deg (x_j(\lambda))$, for $j=1,2,\hdots,p$.
Then, $\mu_1\leq \cdots \leq \mu_q$ and $\epsilon_1\leq \cdots \leq \epsilon_p$ are, respectively, the \emph{left and right minimal indices} of $P(\lambda)$.
\end{definition}

To work in practice with minimal basis we introduce the following definitions, where by the \emph{$i$th row degree} of a matrix polynomial $Q(\lambda)$ we denote the degree of the $i$th row of $Q(\lambda)$ (see \cite[Definition 2.3]{DDMVzigzag}).

\begin{definition} \label{def:rowreduced}
Let $Q(\lambda)\in\mathbb{F}[\lambda]^{m\times n}$ be a matrix polynomial with row degrees $d_1,d_2,\hdots,d_m$. The {\em highest row degree coefficient matrix} of $Q(\lambda)$, denoted by $Q_h$, is the $m \times n$ constant matrix whose $j$th row is the coefficient of $\lambda^{d_j}$ in the $j$th row of $Q(\lambda)$, for $j=1,2,\hdots,m$. 
The matrix polynomial $Q(\lambda)$ is called {\em row reduced} if $Q_h$ has full row rank.
\end{definition}

The following theorem is a useful characterization of minimal bases.
This theorem was originally proved in \cite[Main Theorem-Part 2, p. 495]{Forney}, though the statement we present here can be found in  \cite[Theorem 2.14]{FFP2015}.
\begin{theorem}
\label{thm:minimal_basis}
The rows of a matrix polynomial $Q(\lambda)\in\mathbb{F}[\lambda]^{m\times n}$ are a minimal basis of the rational subspace they span if and only if $Q(\lambda_0) \in \overline{\mathbb{F}}^{m \times n}$ has full row rank for all $\lambda_0 \in \overline{\mathbb{F}}$ and $Q(\lambda)$ is row reduced.
\end{theorem}

\begin{remark} {\rm Most of the minimal bases appearing in this work are arranged as the rows of a matrix. Therefore, throughout the paper, with a slight abuse of notation, we say that an $m\times n$ matrix polynomial (with $m < n$) is a minimal basis if its rows form a minimal basis of the rational subspace they span.}
\end{remark}

The concept of \emph{dual minimal bases} plays an important role in this paper and is introduced in Definition \ref{def:dualminimalbases}. 
\begin{definition}{\rm \cite[Definition 2.10]{DDMVzigzag}} \label{def:dualminimalbases}
Two matrix polynomials $L(\lambda)\in\FF[\lambda]^{m_1\times n}$ and $N(\lambda)\in\FF[\lambda]^{m_2\times n}$ are called \emph{dual minimal bases} if $L(\lambda)$ and $N(\lambda)$ are both minimal bases and they satisfy
$m_1+m_2 = n$ and $L(\lambda)N(\lambda)^T = 0$.
\end{definition}
\begin{remark}
We will sometime say ``$N(\lambda)$ is a minimal basis dual to $L(\lambda)$'', or vice versa, to refer to matrix polynomials $L(\lambda)$ and $N(\lambda)$ as those in Definition \ref{def:dualminimalbases}.
\end{remark}

The following examples  illustrates the concept of dual minimal bases (Example \ref{ex-L-Lamb} can be also found in \cite[Example 2.6]{Fiedler-like}).
The dual minimal bases in these examples will play a crucial role in the developments of this paper.
\begin{example} \label{ex-L-Lamb} {\rm  Consider the following matrix polynomials:
\begin{equation}
\label{eq:Lk}
L_k(\lambda):=\begin{bmatrix}
-1 & \lambda  \\
& -1 & \lambda \\
& & \ddots & \ddots \\
& & & -1 & \lambda  \\
\end{bmatrix}\in\mathbb{F}[\lambda]^{k\times(k+1)},
\end{equation}
and
\begin{equation}
  \label{eq:Lambda}
  \Lambda_k(\lambda)^T :=
\begin{bmatrix}
      \lambda^{k} & \cdots & \lambda & 1
\end{bmatrix} \in \FF[\lambda]^{1\times (k+1)},
\end{equation}
where here and throughout the paper we occasionally omit some, or all, of the zero entries of a matrix. 
Theorem \ref{thm:minimal_basis} guarantees that $L_k(\lambda)$ and $\Lambda_k(\lambda)^T$ are minimal bases. 
Additionally, $L_k(\lambda)\Lambda_k(\lambda)=0$ holds. Therefore, $L_k(\lambda)$ and $\Lambda_k(\lambda)^T$ are dual minimal bases. 
Also, from \cite[Corollary 2.4]{Fiedler-like} and the properties of the Kronecker product we get that $L_k(\lambda) \otimes I_p$ and $\Lambda_k(\lambda)^T \otimes I_p$ are also dual minimal bases.
}
\end{example}
\begin{example} \label{ex-L-Lamb2} {\rm  Consider the following matrix polynomials:
\begin{equation}\label{eq:Lk2}
\widehat{L}_k(\lambda):=
\begin{bmatrix}
0 & -1 & \lambda \\
0 & & -1 & \lambda \\
\vdots & & & \ddots & \ddots \\
0 & & & & -1 & \lambda
\end{bmatrix}\in\FF[\lambda]^{(k-1)\times (k+1)}
\end{equation}
and 
\begin{equation} \label{eq:Lambda2}
\widehat{\Lambda}_k(\lambda)^T :=
\begin{bmatrix}
1 & 0 & \cdots & 0 & 0 \\
0 & \lambda^{k-1} & \cdots & \lambda & 1
\end{bmatrix}\in\FF[\lambda]^{2\times (k+1)}.
\end{equation}
Theorem \ref{thm:minimal_basis} guarantees that $\widehat{L}_k(\lambda)$ and $\widehat{\Lambda}_k(\lambda)^T$ are minimal bases.
Since $\widehat{L}_k(\lambda)\widehat{\Lambda}_k(\lambda)=0$ holds, the matrix polynomials $\widehat{L}_k(\lambda)$ and $\widehat{\Lambda}_k(\lambda)^T$ are dual minimal bases.
Then, from \cite[Corollary 2.4]{Fiedler-like} and the properties of the Kronecker product we get that $\widehat{L}_k(\lambda) \otimes I_p$ and $\widehat{\Lambda}_k(\lambda)^T \otimes I_p$ are also dual minimal bases.
Additionally, notice that although $\rev \widehat{L}_k(\lambda)\rev \widehat{\Lambda}_k(\lambda)=0$ holds, the matrix polynomials $\rev \widehat{L}_k(\lambda)$ and $\rev \widehat{\Lambda}_k(\lambda)$ are not dual minimal bases (since $\rev \widehat{\Lambda}_k(\lambda)$ is not a minimal basis). 
However, we can easily find a dual minimal basis to $\rev \widehat{L}_k(\lambda)$.
Indeed, such a basis may be 
\begin{equation} \label{eq:Lambda3}
\widetilde{\Lambda}_k(\lambda)^T :=
\begin{bmatrix}
1 & 0 & 0 & \cdots & 0 \\
0 & 1 & \lambda & \cdots & \lambda^{k-1}
\end{bmatrix}\in\FF[\lambda]^{2\times (k+1)}.
\end{equation}
Also, from \cite[Corollary 2.4]{Fiedler-like} and the properties of the Kronecker product, we get that $\rev \widehat{\Lambda}_k(\lambda)\otimes I_p$ and $\widetilde{\Lambda}_k(\lambda)^T\otimes I_p$ are dual minimal bases.
}
\end{example}

\medskip

An $m\times m$ matrix polynomial $U(\lambda)$ is said to be \emph{unimodular} if $\det U(\lambda)$ is a nonzero constant.
Two matrix polynomials $P(\lambda)$ and $Q(\lambda)$ are said to be \emph{strictly equivalent} if there exist nonsingular constant matrices $E$ and $F$ such that $EP(\lambda)F=Q(\lambda)$
In addition, the matrix polynomials $P(\lambda)$ and $Q(\lambda)$ are said to be \emph{unimodularly equivalent} if there exist unimodular matrix polynomials $U(\lambda)$ and $V(\lambda)$ such that $U(\lambda)P(\lambda)V(\lambda) = Q(\lambda)$, or \emph{extended unimodularly equivalent}  if  $U(\lambda)\diag(I_s,P(\lambda))V(\lambda)=\diag(I_t,Q(\lambda))$, for some natural numbers $s,t$ (see \cite[Definition 3.2]{IndexSum}).
We recall that strict equivalence preserve  size, degree, and all the spectral structure --finite and infinite-- and all the singular structure of matrix polynomials.
By contrast (extended) unimodular equivalence preserves only (size), dimensions of the left and right null spaces and the finite spectral structure of matrix polynomials.

Finally, we recall the definition of (strong) linearization and (structured) companion form. 
For  linearizations, we will follow the definition in \cite[Definition 3.3(a)]{IndexSum}, which is based on the concept of spectral equivalent matrix polynomials \cite[Definition 3.2(b)]{IndexSum}.
\begin{definition}\label{def:linearization}
A \emph{linearization} of a matrix polynomial $P(\lambda)$  is a pencil $L(\lambda) = \lambda B + A$ such that there exist two unimodular matrix polynomials $U(\lambda)$ and $V(\lambda)$ satisfying
\begin{equation*}
U(\lambda)L(\lambda)V(\lambda) = 
\begin{bmatrix}
I_s & 0 \\ 0 & P(\lambda)
\end{bmatrix},
\end{equation*}
for some natural number $s$.
In addition, the pencil $L(\lambda)$ is said to be a \emph{strong linearization} if $\rev L(\lambda) = \lambda A + B$ is a linearization of $\rev P(\lambda)$.
\end{definition}
We recall that the key property of any strong linearization $L(\lambda)$ of $P(\lambda)$ is that $L(\lambda)$ preserves the finite and infinite eigenstructure of $P(\lambda)$ as well as the dimensions of the right and left null spaces of $P(\lambda)$ (see, for example, \cite{IndexSum}).
On the other hand, it is well known that linearizations may change right and left minimal indices arbitrarily \cite[Theorem 4.11]{IndexSum}.

In numerical applications, it is very important to be able to construct linearizations without performing any arithmetic operation, which may introduce errors that do not exist  in the original problem.
For this reason, in this work not only we are interested in strong linearizations of matrix polynomials, but also in \emph{companion forms} and \emph{structured companion forms} for matrix polynomials.
\begin{definition}{\rm (\cite[Definition 5.1]{IndexSum} and \cite[Definition 7]{Mackey-Fiedler})}\label{def:companion}
A {\rm companion form} for degree-$d$ matrix polynomials is a uniform template for building a pencil $C_P$ from the data of any matrix polynomial $P(\lambda)$ of degree $d$, using no matrix operations on the coefficients of $P(\lambda)$. $C_P$ should be a strong linearization for every  $P(\lambda)$ of degree $d$, regular or singular, over an arbitrary field $\mathbb{F}$.
Additionally, we say that $C_P$ is a {\rm structured companion form} for structure class $\mathcal{S}$ if it is a companion form with the additional property that $C_P\in\mathcal{S}$ whenever $P(\lambda)\in\mathcal{S}$.
\end{definition}
\begin{remark}
Notice that Definition \ref{def:companion} implies that if a strong linearization $\mathcal{L}(\lambda)$ of a matrix polynomial $P(\lambda)=\sum_{k=0}^d P_k\lambda^k$ is constructed using block entries of the form
\begin{equation}\label{eq:M_block_entry}
\lambda B_{ij}+A_{ij} = \lambda \left(\beta_{ij}P_\ell\right) + \alpha_{ij}P_t,
\end{equation}
for some constants $\alpha_{ij},\beta_{ij}$ and natural numbers $\ell,t$, then $\mathcal{L}(\lambda)$ is a companion form for degree-$d$ matrix polynomials..
\end{remark}

\section{Strong minimal bases pencils}\label{sec:pencils}

In this section we revisit the  family of strong minimal bases pencils  recently introduced in \cite{Fiedler-like}.
\begin{definition}{\rm \cite[Definition 3.1]{Fiedler-like}} \label{def:minlinearizations} A matrix pencil
\begin{equation} \label{eq:minbaspencil}
\mathcal{L}(\lambda) = \begin{bmatrix} M(\lambda) & K_2 (\lambda)^T \\ K_1 (\lambda) & 0\end{bmatrix}
\end{equation}
is called a {\em block minimal bases pencil} if $K_1 (\lambda)$ and $K_2(\lambda)$ are both minimal bases.
If, in addition, the row degrees of $K_1 (\lambda)$ are all equal to $1$, the row degrees of $K_2 (\lambda)$ are all equal to $1$, the row degrees of a minimal basis dual to $K_1 (\lambda)$ are all equal, and the row degrees of a minimal basis dual to $K_2 (\lambda)$ are all equal, then $\mathcal{L}(\lambda)$ is called a {\em strong block minimal bases pencil}.
\end{definition}

A surprising property of any (strong) block minimal bases pencil is that it is a (strong) linearization of a certain matrix polynomial expressed in terms of the pencil $\lambda B+A$ and any dual minimal bases of $K_1(\lambda)$ and $K_2(\lambda)$.
More precisely, we have the following theorem.

\begin{theorem}{\rm \cite[Theorem 3.3]{Fiedler-like}}\label{thm:blockminlin}  Let $K_1 (\lambda)$ and $N_1 (\lambda)$ be a pair of dual minimal bases, and let $K_2 (\lambda)$ and $N_2 (\lambda)$ be another pair of dual minimal bases. 
Consider the matrix polynomial
\begin{equation} \label{eq:Qpolinminbaslin}
P(\lambda) := N_2(\lambda) M(\lambda) N_1(\lambda)^T,
\end{equation}
and the block minimal bases pencil $\mathcal{L}(\lambda)$ in \eqref{eq:minbaspencil}. 
Then:
\begin{enumerate}
\item[\rm (a)] $\mathcal{L}(\lambda)$ is a linearization of $P(\lambda)$.
\item[\rm (b)] If $\mathcal{L}(\lambda)$ is a strong block minimal bases pencil, then $\mathcal{L}(\lambda)$ is a strong linearization of $P(\lambda)$, considered as a polynomial with grade $1 + \deg(N_1 (\lambda)) + \deg(N_2 (\lambda))$.
\end{enumerate}
\end{theorem}

From Theorem \ref{def:minlinearizations} we see that given a matrix polynomial $P(\lambda)$ and fixed minimal bases $N_1(\lambda)$ and $N_2(\lambda)$, one can obtain strong linearizations via the pencil \eqref{eq:minbaspencil} provided that the equation \eqref{eq:Qpolinminbaslin} is solved.
This equation can be viewed as a linear equation for the unknown pencil $M(\lambda)$, and it is always consistent as a consequence of the properties of the minimal bases $N_1 (\lambda)$ and $N_2 (\lambda)$ \cite{Fiedler-like}. 
However, despite its consistency, the equation \eqref{eq:Qpolinminbaslin} may be very difficult to solve for arbitrary minimal bases $N_1 (\lambda)$ and $N_2 (\lambda)$.
In Sections \ref{sec:Kronecker_pencils_odd} and \ref{sec:Kronecker_pencils_even} we will see that for certain particular choices of the dual minimal bases $K_1(\lambda),N_1(\lambda)$ and $K_2(\lambda),N_2(\lambda)$ it is, first, easy to obtain solutions $M(\lambda)$ of \eqref{eq:Qpolinminbaslin} and, second, to construct a wide class of linearizations easily constructible from the coefficients of $P(\lambda)$ that are symmetric (resp. skew symmetric) whenever the matrix polynomial is symmetric (resp. skew symmetric). 

We end this section with Lemma \ref{lemma:aux}, which relates eigenvectors, minimal bases and minimal indices of a strong minimal bases pencil \eqref{eq:minbaspencil} with those of the matrix polynomial \eqref{eq:Qpolinminbaslin}.
The results in Lemma \ref{lemma:aux} are immediate consequences of  \cite[Theorem 3.7]{Fiedler-like}, \cite[Lemmas 7.1 and 7.2]{Fiedler-like} and \cite[Remark 7.3]{Fiedler-like}, and will be important in Section \ref{sec:recovery_odd}.
\begin{lemma}\label{lemma:aux}
Let $\mathcal{L}(\lambda)$ be a strong block minimal bases pencil as in \eqref{eq:minbaspencil}, let $N_1(\lambda)$ be a minimal basis dual to $K_1(\lambda)$, and let $P(\lambda)$ be the matrix polynomial defined in \eqref{eq:Qpolinminbaslin}.
Then, the following statements hold.
\begin{itemize}
\item[\rm (a)] Assume that $\mathcal{L}(\lambda)$ is square and regular.
Any right eigenvector of $\mathcal{L}(\lambda)$ with eigenvalue $\lambda_0$ has the form
\[
z=\begin{bmatrix}
N_1(\lambda_0)^T \\ *
\end{bmatrix}x,
\] 
for some right eigenvector $x$ of $P(\lambda)$ with eigenvalue $\lambda_0$,  where by ``$*$'' we denote some matrix polynomial on $\lambda_0$.
\item[\rm (b)] Assume that $\mathcal{L}(\lambda)$ is singular.
\begin{itemize}
\item[\rm (b1)] Any right minimal basis of $\mathcal{L}(\lambda)$ has the form
\[
\left\{
\begin{bmatrix}
N_1(\lambda)^T \\ *
\end{bmatrix}h_1(\lambda),\hdots, \begin{bmatrix}
N_1(\lambda)^T \\ *
\end{bmatrix}h_p(\lambda)
\right\},
\]
where $\{h_1(\lambda),\hdots,h_p(\lambda)\}$ is some right minimal basis of $P(\lambda)$, and where by ``$*$'' we denote some matrix polynomial.
\item[\rm (b2)] If $\epsilon_1\leq \cdots \leq \epsilon_p$ are the right minimal indices of $P(\lambda)$,  then
\[
\epsilon_1+\deg(N_1(\lambda)) \leq \cdots \leq \epsilon_p+\deg(N_1(\lambda)),
\]
are the right minimal indices of $\mathcal{L}(\lambda)$.
\end{itemize}
\end{itemize}
\end{lemma}

\section{Symmetrizable and skew-symmetrizable block Kronecker pencils associated with odd-grade matrix polynomials}\label{sec:Kronecker_pencils_odd}

We begin this section identifying two subfamilies of strong minimal bases pencils that we have named \emph{symmetrizable block Kronecker pencils} and \emph{skew-symmetrizable block Kronecker pencils}.
\begin{definition} \label{def:blockKronlin} Let $L_k (\lambda)$ be the matrix pencil defined in \eqref{eq:Lk}, let $s$ and $n$ be  nonzero natural numbers, let $\sigma\in\{1,-1\}$, and let $\lambda B + A$ be an arbitrary pencil of size $(s+1)n\times (s+1)n$. 
Then any matrix pencil of the form
\begin{equation}
  \label{eq:linearization_general}
  \mathcal{L}(\lambda)=
  \left[
    \begin{array}{c|c}
      \lambda B+A&L_{s}(\lambda)^{T}\otimes I_{n}\\\hline
      \sigma L_{s}(\lambda)\otimes I_{n}&0
      \end{array}
    \right]   \end{equation}
is called a {\rm symmetrizable block Kronecker pencil} if $\sigma=1$, or a {\rm skew-symmetrizable block Kronecker pencil} if $\sigma=-1$. 
Additionally, for simplicity or when the scalar $\sigma$ is not specified, the pencil \eqref{eq:linearization_general} is called a \emph{block Kronecker pencil}.
\end{definition}
\begin{remark}{\rm
In \cite[Definition 5.1]{Fiedler-like} a $(\epsilon,n,\eta,m)$-block Kronecker pencil is defined as a pencil of the form
\[
  \left[
    \begin{array}{c|c}
      \lambda B+A&L_{\eta}(\lambda)^{T}\otimes I_{m}\\\hline
       L_{\epsilon}(\lambda)\otimes I_{n}&0
      \end{array}\right].
\]
So notice that symmetrizable block Kronecker pencils are  particular examples of block Kronecker pencils. 
More precisely, symmetrizable block Kronecker pencils are $(s,n,s,n)$-block Kronecker pencils.
On the other hand, skew-symmetrizable block Kronecker pencils may be obtained by multiplying $(s,n,s,n)$-block Kronecker pencils on the left  by the matrix $E=\diag(I_{(s+1)n},-I_{sn})$, i.e., skew-symmetrizable block Kronecker pencils are strictly equivalent to $(s,n,s,n)$-block Kronecker pencils. 
}
\end{remark}
\begin{remark} {\rm The motivation for the names ``symmetrizable  block Kronecker pencil'' and ``skew-symmetrizable  block Kronecker pencil'' is  the fact that the anti-diagonal blocks of  \eqref{eq:linearization_general} are (up to a sign in the skew-symmetric case) Kronecker products of singular blocks of the Kronecker canonical form of pencils  times identity matrices. Moreover, as we will show in Sections \ref{sec:sym_odd} and \ref{sec:skew_odd}, those families contain infinite symmetric and skew-symmetric strong linearizations for, respectively, symmetric and skew-symmetric matrix polynomials with odd degrees.
}
\end{remark}

From Example \ref{ex-L-Lamb}, it is clear that $\sigma L_k(\lambda)\otimes I_n$ and $\Lambda_k(\lambda)\otimes I_n$ (with $\sigma\in\{1,-1\}$) are a pair of dual minimal bases. 
Therefore, we obtain the following result for symmetrizable or skew-symmetrizable block Kronecker pencils as an immediate corollary of Theorem \ref{thm:blockminlin}.
\begin{theorem}\label{thm:strong}
Let $\mathcal{L}(\lambda)$ be the pencil in \eqref{eq:linearization_general}, and let $\Lambda_k(\lambda)$ be the matrix polynomial in \eqref{eq:Lambda}.
Then $\mathcal{L}(\lambda)$ is a strong linearization of the  matrix polynomial
\begin{equation}\label{eq:condition}
P(\lambda) = (\Lambda_s(\lambda)^T\otimes I_n)(\lambda B+A)(\Lambda_{s}(\lambda)\otimes I_n) \in \mathbb{F}[\lambda]^{n\times n}
\end{equation}
of  grade $2s+1$.
\end{theorem}

\begin{remark}{\rm
Notice in Theorem \ref{thm:strong} that the pencil \eqref{eq:linearization_general} is a strong linearization of the matrix polynomial with odd grade.
This  is the reason why we say that the  block Kronecker pencils introduced in this section are \emph{associated with odd-grade matrix polynomials}.
Note also that grade can not be replaced just by degree in the statement of Theorem \ref{thm:strong} since it may occur that $\deg(P(\lambda))<2s+1$.
}
\end{remark}

A direct matrix multiplication and some elementary manipulations of summations allow us to obtain from Theorem \ref{thm:strong} the conditions on $\lambda B+A$ that guarantee that the pencil $\mathcal{L}(\lambda)$ in \eqref{eq:linearization_general} is a strong linearization of a given matrix polynomial $P(\lambda)$ with odd degree.
These conditions can be expressed using the block Hadamard product and the sum of the block entries notation introduced in Section \ref{sec:intro} using the matrix polynomial $\Gamma_s(\lambda)$ defined by
\begin{equation}\label{eq:Gamma}
\Gamma_s(\lambda) := (\Lambda_s(\lambda)\otimes I_n)(\Lambda_s(\lambda)^T\otimes I_n) =
\begin{bmatrix}
\lambda^{d-1}I_n & \lambda^{d-2}I_n & \cdots & \lambda^sI_n \\
\lambda^{d-2}I_n & & \iddots & \vdots\\
\vdots & \iddots & & \lambda I_n \\
\lambda^sI_n & \cdots & \lambda I_n& I_n
\end{bmatrix},
\end{equation}
which has a block Hankel structure.
\begin{corollary} \label{cor:givenPblockKron}
Let $P(\lambda) = \sum_{k=0}^d P_k \lambda^k \in \FF[\lambda]^{n \times n}$ with odd degree $d$, let $\mathcal{L}(\lambda)$ be the pencil in \eqref{eq:linearization_general} with $s=(d-1)/2$, let $\Lambda_k(\lambda)$ and $\Gamma_k(\lambda)$ be the matrix polynomials in \eqref{eq:Lambda} and \eqref{eq:Gamma}, respectively. Partition $A$ and $B$ into $(s+1)\times (s +1)$ blocks each of size $n\times n$. let us Denote these blocks by $A_{ij}, B_{ij} \in\FF^{n\times n}$  for $i=1,\hdots,s+1$ and $j=1,\hdots,s +1$.
Then, any of the the following conditions implies that the pencil $\mathcal{L}(\lambda)$ is a strong linearization of $P(\lambda)$.
\begin{enumerate}
\item[\rm (a)]
 \begin{equation}\label{eq:condition_coeff}
   \sum_{i+j=d+2-k} B_{ij} + \sum_{i+j=d+1-k} A_{ij}=P_k, \quad \mbox{for $k=0,1,\hdots,d$,}
 \end{equation}
\item[\rm (b)]
\begin{equation}\label{eq:condition_M}
P(\lambda) = 
\su\left( (\lambda B+A) \odot 
\Gamma_s(\lambda) \right).
\end{equation}
\end{enumerate}
\end{corollary}

The advantage of \eqref{eq:condition_M} over \eqref{eq:condition_coeff} is that only by inspection we may know in advance the power of $\lambda$ that multiply each block entry $\lambda B_{ij}+A_{ij}$ when \eqref{eq:condition_M} is expanded in the monomial basis.
We illustrate this in the following example,  where we check that three different  block Kronecker pencils that are, indeed, strong linearizations of degree-5 matrix polynomials.
\begin{example}
Let $P(\lambda)=\sum_{k=0}^5 P_k \lambda^k\in\FF[\lambda]^{n\times n}$, and let $\sigma \in\{-1,1\}$.
To obtain strong linearizations of $P(\lambda)$ from  block Kronecker pencils \eqref{eq:linearization_general} we need to find pencils $\lambda B+A$ satisfying \eqref{eq:condition_M}.
We can obtain one just noticing that
\begin{align*}
&\su \left(
\begin{bmatrix}
\lambda P_5 & \lambda P_4 & 0 \\
0 & P_2 & 0 \\
\lambda P_3 & 0 & \lambda P_1+P_0
\end{bmatrix}\odot
\begin{bmatrix}
\lambda^4I_n & \lambda^3 I_n & \lambda^2 I_n \\
\lambda^3 I_n & \lambda^2 I_n & \lambda I_n \\
\lambda^2 I_n & \lambda I_n & I_n
\end{bmatrix} 
\right) = \\
&\su \left(
\begin{bmatrix}
\lambda^5 P_5  & \lambda^4P_4 & 0 \\
0 & \lambda^2 P_2 & 0 \\
\lambda^3 P_3 & 0 & \lambda P_1+P_0
\end{bmatrix}
\right) = P(\lambda)
\end{align*}
clearly holds.
The above equation, together with Corollary \ref{cor:givenPblockKron}(b), implies that the block Kronecker pencil
\begin{equation}\label{eq:ex_1}
\left[
\begin{array}{ccc|cc}
\lambda P_5 & \lambda P_4 & 0 & -I_n & 0 \\
0 & P_2 & 0 & \lambda I_n & -I_n \\
\lambda P_3 & 0 & \lambda P_1+P_0 & 0 & \lambda I_n \\ \hline
-\sigma I_n & \sigma\lambda I_n & 0 & 0 & 0 \\
0 & -\sigma I_n & \sigma \lambda I_n & 0 & 0
\end{array}
\right]
\end{equation}
is a strong linearization of $P(\lambda)$.
We might also have considered the following equation
\begin{align*}
&\su \left(
\begin{bmatrix}
\lambda P_5+P_4 & P_3 & 0 \\
0 & 0 & \lambda P_2 \\
0 & P_1 & P_0
\end{bmatrix}\odot
\begin{bmatrix}
\lambda^4I_n & \lambda^3 I_n & \lambda^2 I_n \\
\lambda^3 I_n & \lambda^2 I_n & \lambda I_n \\
\lambda^2 I_n & \lambda I_n & I_n
\end{bmatrix} 
\right) = \\
&\su \left(
\begin{bmatrix}
\lambda^5 P_5+\lambda^4 P_4 & \lambda^3 P_3 & 0 \\
0 & 0 & \lambda^2 P_2 \\
0 & \lambda P_1 & P_0
\end{bmatrix}
\right) = P(\lambda)
\end{align*}
and obtain, from Corollary \ref{cor:givenPblockKron}(b), that the  block Kronecker pencil
\begin{equation}\label{eq:ex_2}
\left[
\begin{array}{ccc|cc}
\lambda P_5+P_4 & P_3 & 0 & -I_n & 0 \\
0 & 0 & \lambda P_2 & \lambda I_n & -I_n \\
0 & P_1 & P_0 & 0 & \lambda I_n \\ \hline
-\sigma I_n & \sigma\lambda I_n & 0 & 0 & 0 \\
0 & -\sigma I_n & \sigma \lambda I_n & 0 & 0
\end{array}
\right],
\end{equation}
is also strong linearization of $P(\lambda)$.
Note that every block matrix of the pencils \eqref{eq:ex_1} and \eqref{eq:ex_2} is either $\pm I_n$ or $P_k$.
In fact, they are companion forms for matrix polynomials of degree 5 (see Definition \ref{def:companion}).
However, it is possible to construct strong linearizations with other kinds of block entries.
For example, consider the following equation
\begin{align*}
&\su \left(
\begin{bmatrix}
\lambda P_5 & E & -\lambda E \\
\lambda P_4 & \lambda F & \lambda P_2 \\
\lambda (P_3-F) & P_1 & P_0
\end{bmatrix}\odot
\begin{bmatrix}
\lambda^4I_n & \lambda^3 I_n & \lambda^2 I_n \\
\lambda^3 I_n & \lambda^2 I_n & \lambda I_n \\
\lambda^2 I_n & \lambda I_n & I_n
\end{bmatrix} 
\right) = \\
&\su \left(
\begin{bmatrix}
\lambda^5 P_5  & \lambda^3 E & -\lambda^3 E \\
\lambda^4P_4 & \lambda^3 F & \lambda^2 P_2 \\
\lambda^3 (P_3-F) & \lambda P_1 & P_0
\end{bmatrix}
\right) = P(\lambda),
\end{align*}
where $E$ and $F$ are arbitrary $n\times n$ matrices. 
Then, the previous equation and  Corollary \ref{cor:givenPblockKron}(b) imply that the pencil
\begin{equation}\label{eq:ex_3}
\left[
\begin{array}{ccc|cc}
\lambda P_5 & E & -\lambda E & -I_n & 0 \\
\lambda P_4 & \lambda F & \lambda P_2 & \lambda I_n & -I_n \\
\lambda (P_3-F) & P_1 & P_0 & 0 & \lambda I_n \\ \hline
-\sigma I_n & \sigma\lambda I_n & 0 & 0 & 0 \\
0 & -\sigma I_n & \sigma \lambda I_n & 0 & 0
\end{array}
\right],
\end{equation}
is a strong linearization of $P(\lambda)$.
If $F$ is a nonzero matrix, then the pencil \eqref{eq:ex_3} is not a companion form due to the block entry $\lambda (P_3-F)$.
\end{example}

In this section, we showed how to obtain strong linearizations from symmetrizable or skew-symmetrizable block Kronecker pencils for an odd-degree matrix polynomial $P(\lambda)$.
In the following subsections, we present a methodology for obtaining from these families of pencils linearizations that are symmetric or skew-symmetric whenever $P(\lambda)$ is, respectively, symmetric or skew-symmetric.

\subsection{Symmetric linearizations for odd-degree symmetric matrix polynomials}\label{sec:sym_odd}
In this section we characterize all the  structure-preserving strong linearizations for symmetric matrix polynomials with odd degree that can be obtained from the family of symmetrizable block-Kronecker pencils, i.e., from pencils of the form
\begin{equation}
  \label{eq:linearization_sym}
  \mathcal{L}(\lambda)=
  \left[
    \begin{array}{c|c}
      \lambda B+A&L_{s}(\lambda)^{T}\otimes I_{n}\\\hline
      L_{s}(\lambda)\otimes I_{n}&0
      \end{array}
    \right],  
\end{equation}
satisfying the equations in \eqref{eq:condition_coeff} or, equivalently, satisfying \eqref{eq:condition_M}.
    
Notice that the symmetrizable block Kronecker pencil above is symmetric if and only if the pencil $\lambda B+A$ is symmetric. 
For this reason, we obtain in Theorem \ref{thm:solving_sym} all the symmetric pencils $\lambda B+A$ satisfying \eqref{eq:condition_coeff} for a symmetric matrix polynomial $P(\lambda)$.
But first, before addressing this general result,  we start illustrating our approach in Example \ref{ex:sym_odd1}, where we outline the procedure to solve \eqref{eq:condition_coeff} in a small case, namely, for a matrix polynomial of degree $d=5$.
\begin{example}\label{ex:sym_odd1}
In this example we obtain a parametrization of all the symmetric block Kronecker pencils that are symmetric strong linearizations of a degree-5 symmetric matrix polynomial.
Let $P(\lambda)=\sum_{k=0}^5 P_k \lambda^k\in \FF[\lambda]^{n\times n}$ be any degree-5 matrix polynomial. Consider symmetric block Kronecker pencils as in \eqref{eq:linearization_sym} with $s=2$, and  us partition the $3n\times 3n$ matrices $A$ and $B$ into $n\times n$ blocks, denoted by $A_{ij}$ and $B_{ij}$ for $i,j=1,2,3$.
The first step is to consider only pencils $\lambda B+A$ of the form
\[
\lambda B+A = 
\begin{bmatrix}
\lambda B_{11}+A_{11} &  \lambda B_{12}+A_{12}  & \lambda B_{13}+A_{13} \\
\lambda B_{12}^T+A_{12}^T & \lambda B_{22}+A_{22} & \lambda B_{23}+A_{23} \\
\lambda B_{13}^T+A_{13}^T & \lambda B_{23}^T+A_{23}^T & 
\lambda B_{33}+A_{33}
\end{bmatrix}.
\]
The second step is to solve \eqref{eq:condition_coeff} for a pencil $\lambda B+A$ of the above form.
In this case, we obtain the following underdetermined linear system of equations
\begin{align}\label{eq:system_example_sym}
\begin{split}
B_{11} &= P_5, \\
B_{12}+B_{12}^T+A_{11}&=P_4, \\
B_{13}+B_{22}+B_{13}^T+A_{12}+A_{12}^T &= P_3, \\
B_{23}+B_{23}^T + A_{13}+A_{22}+A_{13}^T &= P_2, \\
B_{33}+A_{23}+A_{23}^T &= P_1, \quad \mbox{and} \\
A_{33} &= P_0,
\end{split}
\end{align}
with $6n^2$  equations and $12n^2$ unknowns.
The linear system \eqref{eq:system_example_sym} is consistent with $6n^2$ degrees of freedom that are not difficult to describe.
Indeed, we may take the entries of the matrices of the upper off-diagonal blocks of $A$ and $B$ (i.e., $A_{ij}$ and $B_{ij}$ with $i<j$) as the $6n^2$ free parameters, and, then, set
\begin{align}\label{eq:system_example_sym2}
\begin{split}
&B_{11} := P_5,\\
&A_{11} := P_4-(B_{12}+B_{12}^T), \\
&B_{22} := P_3-(B_{13}+B_{31}^T)-(A_{12}+A_{12}^T),\\
&A_{22} := P_2-(B_{23}+B_{23}^T) - (A_{13}+A_{13}^T), \\
&B_{33} := P_1-(A_{23}+A_{23}^T), \quad \mbox{and} \\
&A_{33} := P_0.
\end{split}.
\end{align}
Finally, notice that the matrices $A_{ii}$ and $B_{ii}$ are symmetric if and only if the matrix polynomial $P(\lambda)$ is symmetric. 
This implies that when $P(\lambda)$ is symmetric, the equations in \eqref{eq:system_example_sym2}  describe the general symmetric pencil  solution $\lambda B+A$ of \eqref{eq:condition_coeff}.
Therefore, when the matrix polynomial $P(\lambda)$ is symmetric, from (a) in Corollary \ref{cor:givenPblockKron}, we obtain that the symmetrizable block Kronecker pencils of the form 
\begin{equation*}
\left[\begin{array}{ccc|cc}
\lambda B_{11}+A_{11} & \lambda B_{12}+A_{12} &\lambda B_{13} + A_{13} & -I_n & 0\\
\lambda B_{12}^T+A_{12}^T &\lambda B_{22}+A_{22}
& \lambda B_{23}+A_{23} & \lambda I_n &  -I_n \\
\lambda B_{13}^T+A_{13}^T & \lambda B_{23}^T+A_{23}^T &\lambda B_{33} + A_{33} & 0 & \lambda I_n \\ \hline
-I_n & \lambda I_n &0 & 0 & 0 \\
0 & -I_n & \lambda I_n & 0 & 0 \\
\end{array}\right],
\end{equation*}
are symmetric strong linearization of $P(\lambda)$, where the entries of the matrices $A_{12},A_{13},A_{23},B_{12},B_{13},B_{23}$ are arbitrary, and the diagonal block matrices  are as in \eqref{eq:system_example_sym2}.
\end{example}

Theorem \ref{thm:solving_sym} is one of the main results in this section.
It provides the general symmetric pencil solution $\lambda B+A$ of the set of equations in \eqref{eq:condition_coeff} when the matrix polynomial $P(\lambda)$ is symmetric. 
\begin{theorem}\label{thm:solving_sym}
Let $P(\lambda)=\sum_{k=0}^d P_k \lambda^k\in \FF[\lambda]^{n\times n}$ be a symmetric matrix polynomial with odd degree. Let $s=(d-1)/2$, and $\lambda B+A\in\FF[\lambda]^{(s+1)n\times (s+1)n}$.  Partition the matrices $A$ and $B$ into $n\times n$ blocks, denoted by $A_{ij}$ and $B_{ij}$ for $i,j=1,2,\hdots,s+1$. 
Then, the symmetric pencil solution $\lambda B+A$ of \eqref{eq:condition_coeff} is obtained setting
\begin{equation}\label{eq:B_sym}
B_{kk} := P_{d-2k+2}- \sum_{\substack{i+j=2k \\ i<j}}\left( B_{ij}+B_{ij}^T \right) 
 -\sum_{\substack{i+j=2k-1 \\ i<j}}\left(A_{ij}+A_{ij}^T\right),
\end{equation}
and
\begin{equation}\label{eq:A_sym}
A_{kk} := P_{d-2k+1}- \sum_{\substack{i+j=2k+1 \\ i<j}}\left( B_{ij}+B_{ij}^T \right) 
 -\sum_{\substack{i+j=2k \\ i<j}}\left(A_{ij}+A_{ij}^T\right),
\end{equation}
for $k=1,2,\hdots,s+1$, where $A_{ji} = A_{ij}^T$ and $B_{ji} = B_{ij}^T$, for $i<j$, and where the entries of the upper off-diagonal blocks of $A$ and $B$ are free parameters.
\end{theorem}
\begin{proof}
We proceed similarly to the strategy outlined in Example \ref{ex:sym_odd1}.
Since we want to obtain the symmetric pencils satisfying \eqref{eq:condition_coeff}, necessarily those pencils satisfy $A_{ij}^T=A_{ji}$ and $B_{ij}^T=B_{ji}$, for $i<j$.
So the first step is to find all the solutions of \eqref{eq:condition_coeff} where the pencil $\lambda B+A$ is of the form
\[
\left[\begin{smallmatrix}
\lambda B_{11}+A_{11} & \lambda B_{12}+A_{12} & \lambda B_{13}+A_{13} & \cdots & \lambda B_{1,s+1}+A_{1,s+1} \\
\lambda B_{12}^T+A_{12}^T & \lambda B_{22}+A_{22} & \lambda B_{23}+A_{23} & \cdots & \lambda B_{2,s+1}+A_{2,s+1} \\
\lambda B_{13}^T+A_{13}^T & \lambda B_{23}^T+A_{23}^T & \lambda B_{33}+A_{33} & \cdots & \lambda B_{3,s+1}+A_{3,s+1}\\
\vdots & \vdots & \vdots & \ddots & \vdots \\
\lambda B_{1,s+1}^T+A_{1,s+1}^T & \lambda B_{2,s+1}^T+A_{2,s+1}^T & \lambda B_{3,s+1}^T+A_{3,s+1}^T & \cdots & \lambda B_{s+1,s+1}+A_{s+1,s+1}
\end{smallmatrix}\right],
\]
and $P(\lambda)$ is any $n\times n$ matrix polynomial of degree $d$ (symmetric or non symmetric).
In this situation, the equations in \eqref{eq:condition_coeff} give rise to a linear system of $(d+1)n^2$ equations with $(s+1)(s+2)n^2$ unknowns.
To solve this linear system, the entries of the off-diagonal block entries $A_{ij}$ and $B_{ij}$, for $i<j$, may be taken as $s(s+1)n^2$ free parameters.
Then, it is straightforward to check that the equations in \eqref{eq:condition_coeff} hold if and only if we set the diagonal block entries $B_{kk}$ and $A_{kk}$  as in \eqref{eq:B_sym} and \eqref{eq:A_sym}.
Finally, just by inspection, it is clear that the diagonal blocks $B_{kk}$ and $A_{kk}$ in \eqref{eq:B_sym} and \eqref{eq:A_sym}, for $k=1,2,\dots,s+1$, are symmetric if and only if the  matrix polynomial $P(\lambda)$ is symmetric.
This implies that when $P(\lambda)$ is symmetric, the equations \eqref{eq:B_sym} and \eqref{eq:A_sym} describe the general symmetric pencil  solution $\lambda B+A$ of the equations in \eqref{eq:condition_coeff}.
\end{proof}

From Theorem \ref{thm:strong} we get that the set of symmetrizable block Kronecker pencils  with  pencils $\lambda B+A$ as in Theorem \ref{thm:solving_sym} are symmetric strong linearizations of the symmetric matrix polynomial $P(\lambda)$.
This set provides a quite large arena of symmetric strong linearizations of symmetric matrix polynomials.
However, it is worth mentioning that not all its  pencils are  companion forms according to Definition \ref{def:companion}.
We illustrate this in Example \ref{ex:sym_odd4}.
\begin{example}\label{ex:sym_odd4}
Let $P(\lambda)=\sum_{k=0}^5 P_k\lambda^k\in\FF[\lambda]^{n\times n}$ be a symmetric matrix polynomial.
The following block Kronecker pencil
\[
\left[
\begin{array}{ccc|cc}
\lambda P_5+P_4 & 0 & E & -I_n & 0 \\
0 & \lambda P_3+P_2-E-E^T & 0 & \lambda I_n & -I_n \\
E^T & 0 & \lambda P_1+P_0 & 0 & \lambda I_n \\ \hline
-I_n & \lambda I_n & 0 & 0 & 0 \\
0 & -I_n & \lambda I_n & 0 & 0
\end{array}
\right],
\] 
where $E$ is any $n\times n$ non skew-symmetric matrix, is a strong linearization of $P(\lambda)$ (since its corresponding pencil $\lambda B+A$ satisfies \eqref{eq:condition_coeff}).
However, note that due to the block entry $\lambda P_3+P_2-E-E^T$, it is not a companion form.
\end{example}

In the remaining of this section, we present some illuminating examples of  block Kronecker pencils that are symmetric companion forms for symmetric odd-degree matrix polynomials. 
That the  block Kronecker pencils in Examples \ref{ex:sym_odd2} and \ref{ex:sym_odd3} are strong linearizations can be verified either checking that the set of the equations in \eqref{eq:condition_coeff} are satisfied, or, more simple, that the equation \eqref{eq:condition_M} is satisfied.
\begin{example}\label{ex:sym_odd2}
We construct here  two different symmetric companion forms of symmetric odd-degree matrix polynomials.
For illustrative purposes, we focus on symmetric degree-7 matrix polynomials  $P(\lambda)=\sum_{k=0}^7 P_k\lambda^k\in\FF[\lambda]^{n\times n}$.
The extension of these companion forms to a companion form for any other odd grade is straightforward.
Our first example is based on the (possibly) simplest choice of a symmetric pencil $\lambda B+A$ satisfying \eqref{eq:condition_M}, namely, $\diag(\lambda P_7+P_6,\lambda P_5+P_4,\lambda P_3+P_2,\lambda P_1+P_0)$.
With this choice for $\lambda B+A$ we obtain the following  block Kronecker pencil
\[
\mathcal{L}_1(\lambda)=
\left[\begin{array}{cccc|ccc}
\lambda P_7+P_6 & 0 & 0 & 0 & -I_n & 0 & 0 \\
0 & \lambda P_5+P_4 & 0 & 0 & \lambda I_n & -I_n & 0 \\
0 & 0 & \lambda P_3+P_2 & 0 & 0 & \lambda I_n & -I_n \\
0 & 0 & 0 & \lambda P_1+P_0 & 0 & 0 & \lambda I_n \\ \hline
-I_n & \lambda I_n & 0 & 0 & 0 & 0 & 0 \\
0 & -I_n & \lambda I_n & 0 & 0 & 0 & 0 \\
0 & 0 & -I_n & \lambda I_n & 0 & 0 & 0
\end{array}\right].
\]
Remarkably, a permuted version of the above pencil has appeared before in the context of Fiedler pencils with repetitions \cite{Greeks2,FPR1}.
Indeed, it is not difficult to show that there exists a permutation matrix $P$ such that 
\[
P^T\mathcal{L}_1(\lambda)P=
\begin{bmatrix}
\lambda P_7+P_6 & -I_n & 0 & 0 & 0 & 0 & 0 \\
-I_n & 0 & \lambda I_n & 0 & 0 & 0 & 0 \\
0 & \lambda I_n & \lambda P_5+P_4 & -I_n & 0 & 0 & 0 \\
0 & 0 & -I_n & 0 & \lambda I_n & 0 & 0 \\
0 & 0 & 0 & \lambda I_n & \lambda P_3+P_2 & -I_n & 0 \\
0 & 0 & 0 & 0 & -I_n & 0 & \lambda I_n \\
0 & 0 & 0 & 0 & 0 & \lambda I_n & \lambda P_1+P_0
\end{bmatrix},
\]
which is the well-known block-tridiagonal symmetric companion form first introduced in \cite{Greeks2}.
Our second example shows a symmetric companion form that has not appeared previously in the literature (to the knowledge of the authors).
This companion form is 
\[
\mathcal{L}_2(\lambda)=
\left[\begin{array}{cccc|ccc}
\lambda P_7-P_6 & \lambda P_6 & 0 & 0 & -I_n & 0 & 0 \\
\lambda P_6 & \lambda P_5-P_4 & \lambda P_4 & 0 & \lambda I_n & -I_n & 0 \\
0 & \lambda P_4 & \lambda P_3-P_2 & \lambda P_2 & 0 & \lambda I_n & -I_n \\
0 & 0 & \lambda P_2 & \lambda P_1+P_0 & 0 & 0 & \lambda I_n \\\hline
-I_n & \lambda I_n & 0 & 0 & 0 & 0 & 0 \\
0 & -I_n & \lambda I_n & 0 & 0 & 0 & 0 \\
0 & 0 & -I_n & \lambda I_n & 0 & 0 & 0 
\end{array}\right],
\]
which also can be permuted to obtain a  low-bandwidth (block-pentadiagonal) symmetric companion form, i.e., there exists a permutation matrix $Q$ such that
\[
Q^T\mathcal{L}_2(\lambda)Q =
\begin{bmatrix}
\lambda P_7-P_6 & -I_n & \lambda P_6 & 0 & 0 & 0 & 0 \\
-I_n & 0 & \lambda I_n & 0 & 0 & 0 & 0 \\
\lambda P_6 & \lambda I_n & \lambda P_5-P_4 & -I_n & \lambda P_4 & 0 & 0 \\
0 & 0 & -I_n & 0 & \lambda I_n & 0 & 0 \\
0 & 0 & \lambda P_4 & \lambda I_n & \lambda P_3-P_2 & -I_n & \lambda P_2 \\
0 & 0 & 0 & 0 & -I_n & 0 & \lambda I_n \\
0 & 0 & 0 & 0 & \lambda P_2 & \lambda I_n & \lambda P_1+P_0
\end{bmatrix}.
\]
\end{example}

It is well-known that the family of Fiedler pencils with repetition (FPR) provides a convenient arena in which to look for structured linearizations of structured polynomials \cite{Greeks2,FPR1,FPR2,FPR3,Greeks3}.
The characterization of all the FPR that are symmetric when the matrix polynomial is has been carried out in \cite{FPR1,Curlett}.
In the following example we show that the examples of symmetric companion forms obtained from FPR in \cite[Example 8]{Greeks3} are, indeed,  block Kronecker pencils (modulo a permutation).
\begin{example}\label{ex:sym_odd3}
Let us consider the symmetric companion forms for symmetric matrix polynomials  $P(\lambda)=\sum_{k=0}^5 P_k\lambda^k\in\FF[\lambda]^{n\times n}$ in {\rm \cite[Example 8]{Greeks}}, that is, the pencils
\[
L_3^\prime(\lambda)=
\begin{bmatrix}
0 & \lambda I_n & -I_n & 0 & 0 \\
\lambda I_n & -\lambda P_1+P_0 & P_1 & 0 & 0 \\
-I_n & P_1 & \lambda P_3+P_2 & \lambda P_4 & \lambda I_n \\
0 & 0 & \lambda P_4 & \lambda P_5-P_4 & -I_n \\
0 & 0 & \lambda I_n & -I_n & 0
\end{bmatrix}
\]
and
\[
L_5^\prime(\lambda)=
\begin{bmatrix}
0 & 0 & \lambda I_n & -I_n & 0 \\
0 & 0 & 0 & \lambda I_n & -I_n \\
\lambda I_n & 0 & -\lambda P_1+P_0 & -\lambda P_2+P_1 & P_2\\
-I_n & \lambda I_n & -\lambda P_2+P_1 & -\lambda P_3+P_2 & P_3 \\
0 & -I_n & P_2 & P_3 & \lambda P_5+P_4
\end{bmatrix}.
\]  
Then, it is easy to check that there exist permutation matrices $P$ and $Q$ such that
\[
P^T L_3^\prime(\lambda) P = 
\left[\begin{array}{ccc|cc}
\lambda P_5-P_4 & \lambda P_4 & 0 & -I_n & 0 \\
\lambda P_4 & \lambda P_3+P_2 & P_1 & \lambda I_n & -I_n \\
0 & P_1 & -\lambda P_1+P_0 & 0 & \lambda I_n\\ \hline
-I_n & \lambda I_n & 0 & 0 & 0 \\
0 & -I_n & \lambda I_n & 0 & 0
\end{array}\right]
\]
and
\[
Q^T L_5^\prime(\lambda) Q = 
\left[\begin{array}{ccc|cc}
\lambda P_5+P_4 & P_3 & P_2 & -I_n & 0 \\
P_3 & -\lambda P_3+P_2 & -\lambda P_2+P_1 & \lambda I_n & -I_n \\
P_2 & -\lambda P_2+P_1 & -\lambda P_1+P_0 & 0 & \lambda I_n \\ \hline
-I_n & \lambda I_n & 0 & 0 & 0 \\
0 & -I_n & \lambda I_n & 0 & 0
\end{array}\right],
\]
which, clearly, are block Kronecker pencils satisfying \eqref{eq:condition_M}.
This examples open the following question: are all the symmetric companion forms obtained from FPR in \cite{FPR1} permuted block Kronecker pencils? The answer of this question will be the subject of future work.
\end{example}

\subsection{Skew-symmetric linearizations for odd-degree skew-symmetric matrix polynomials}\label{sec:skew_odd}

In this section we characterize all the  structure-preserving strong linearizations for skew-symmetric matrix polynomials with odd degree that can be obtained from the family of skew-symmetrizable block-Kronecker pencils, i.e., from pencils of the form
\begin{equation*}
  \label{eq:linearization_skew-sym}
  \mathcal{L}(\lambda)=
  \left[
    \begin{array}{c|c}
      \lambda B+A&L_{s}(\lambda)^{T}\otimes I_{n}\\\hline
      -L_{s}(\lambda)\otimes I_{n}&0
      \end{array}
    \right],
\end{equation*}
satisfying the equations in \eqref{eq:condition_coeff} or, equivalently, satisfying \eqref{eq:condition_M}.
    
Notice that the skew-symmetrizable block Kronecker pencil above is skew-symmetric if and only if the pencil $\lambda B+A$ is also skew-symmetric. 
In Theorem \ref{thm:solving_skew-sym}, we obtain a parametrization of all the skew-symmetric pencils $\lambda B+A$ satisfying \eqref{eq:condition_coeff} for a skew-symmetric matrix polynomial $P(\lambda)$.
Its proof is analogous to the one for Theorem \ref{thm:solving_sym}, so it is omitted.
\begin{theorem}\label{thm:solving_skew-sym}
Let $P(\lambda)=\sum_{k=0}^d P_k\lambda^k \in\FF[\lambda]^{n\times n}$ be a skew-symmetric odd-degree matrix polynomial, let $s=(d-1)/2$, let $\lambda B+A\in\FF[\lambda]^{(s+1)n\times (s+1)n}$, and  let us  partition the matrices $A$ and $B$ into $n\times n$ blocks, denoted by $A_{ij}$ and $B_{ij}$ for $i,j=1,2,\hdots,s+1$. 
Then, the skew-symmetric pencil solution $\lambda B+A$ of \eqref{eq:condition_coeff} is obtained setting
\begin{equation*}\label{eq:B_skew-sym}
B_{kk} = P_{d-2k+2}- \sum_{\substack{i+j=2k \\ i<j}}\left( B_{ij}-B_{ij}^T \right) 
 -\sum_{\substack{i+j=2k-1 \\ i<j}}\left(A_{ij}-A_{ij}^T\right),
\end{equation*}
and
\begin{equation*}\label{eq:A_skew-sym}
A_{kk} = P_{d-2k+1}- \sum_{\substack{i+j=2k+1 \\ i<j}}\left( B_{ij}-B_{ij}^T \right) 
 -\sum_{\substack{i+j=2k \\ i<j}}\left(A_{ij}-A_{ij}^T\right),
\end{equation*}
for $k=1,2,\hdots,s+1$, where  $A_{ji} = -A_{ij}^T$ and $B_{ji} = -B_{ij}^T$, for $i<j$, and where the entries of the upper off-diagonal blocks of $A$ and $B$ are free parameters.
\end{theorem}

To illustrate Theorem \ref{thm:solving_skew-sym}, in the following example we show a parametrization of the set of  Kronecker pencils that are skew-symmetric strong linearizations of a skew-symmetric matrix polynomial of degree 5.
\begin{example}\label{ex:sym_even1}
Let $P(\lambda)=\sum_{k=0}^5 P_k\lambda^d\in\FF[\lambda]^{n\times n}$ be a skew-symmetric matrix polynomial.
Then, from Theorem \ref{thm:solving_skew-sym} and Corollary \ref{cor:givenPblockKron}(a) we obtain that the set of skew-symmetric block Kronecker pencils of the form 
 \[
\left[\begin{array}{ccc|cc}
\lambda B_{11}+A_{11} & \lambda B_{12}+A_{12} &\lambda B_{13} + A_{13} & -I_n & 0\\
-\lambda B_{12}^T-A_{12}^T &\lambda B_{22}+A_{22}
& \lambda B_{23}+A_{23} & \lambda I_n &  -I_n \\
-\lambda B_{13}^T-A_{13}^T & -\lambda B_{23}^T-A_{23}^T & \lambda B_{33} + A_{33} & 0 & \lambda I_n \\ \hline
I_n & -\lambda I_n &0 & 0 & 0 \\
0 & I_n & -\lambda I_n & 0 & 0 \\
\end{array}\right],
\]
where the matrices $A_{12},A_{13},A_{23},B_{12},B_{13},B_{23}$ are arbitrary, and where the diagonal block matrices  are given by
\begin{align*}
&B_{11} = P_5,\\
&A_{11} = P_4-(B_{12}-B_{12}^T), \\
&B_{22} = P_3-(B_{13}-B_{31}^T)-(A_{12}-A_{12}^T),\\
&A_{22} = P_2-(B_{23}-B_{23}^T) - (A_{13}-A_{13}^T), \\
&B_{33} = P_1-(A_{23}-A_{23}^T), \quad \mbox{and} \\
&A_{33} = P_0,
\end{align*}
are skew-symmetric strong linearizations of $P(\lambda)$.
\end{example}

The set of  block Kronecker pencils with pencils $\lambda B+A$ as in Theorem \ref{thm:solving_skew-sym} provides a large arena of skew-symmetric strong linearizations (an infinite family with $s(s+1)n^2$ free parameters).
However, similarly to what we have noticed in Example \ref{ex:sym_odd4} not all its pencils are skew-symmetric companion forms.
We illustrate this in Example \ref{ex:skew-sym_odd4}.
\begin{example}\label{ex:skew-sym_odd4}
Let $P(\lambda)=\sum_{k=0}^5 P_k\lambda^d\in\FF[\lambda]^{n\times n}$ be a skew-symmetric matrix polynomial.
The following block Kronecker pencil
\[
\left[
\begin{array}{ccc|cc}
\lambda P_5+P_4 & 0 & E & -I_n & 0 \\
0 & \lambda P_3+P_2-E+E^T & 0 & \lambda I_n & -I_n \\
-E^T & 0 & \lambda P_1+P_0 & 0 & \lambda I_n \\ \hline
I_n & -\lambda I_n & 0 & 0 & 0 \\
0 & I_n & -\lambda I_n & 0 & 0
\end{array},
\right]
\]
where $E$ is any $n\times n$ non symmetric matrix, is a strong linearizations of $P(\lambda)$ (since its corresponding pencil $\lambda B+A$ satisfies \eqref{eq:condition_coeff}).
However, notice that due to the block entry $\lambda P_3+P_2-E+E^T$, it is not a companion form.
\end{example}

We end this section with some examples of skew-symmetric strong linearizations obtained from the set of block Kronecker pencils. 
First, in Example \ref{ex:skew-sym_odd2}, we show that the well-known block-tridiagonal skew-symmetric companion form in \cite{Greeks,Skew} is, modulo a permutation, a block Kronecker pencil.
\begin{example}\label{ex:skew-sym_odd2}
In this example, we  consider a skew-symmetric matrix polynomial $P(\lambda)=\sum_{k=0}^7P_k\lambda^k \in\FF[\lambda]^{n\times n}$, but the result presented here can be easily generalized to any odd-degree matrix polynomial.
The simplest choice of a skew-symmetric pencil $\lambda B+A$ satisfying \eqref{eq:condition_coeff} (or \eqref{eq:condition_M}) is, probably, $\diag(\lambda P_7+P_6,\lambda P_5+P_4,\lambda P_3+P_2,\lambda P_1+P_0)$.
With this choice for $\lambda B+A$ we obtain that the following  block Kronecker pencil
\[
\mathcal{L}_1(\lambda)=
\left[\begin{array}{cccc|ccc}
\lambda P_7+P_6 & 0 & 0 & 0 & -I_n & 0 & 0 \\
0 & \lambda P_5+P_4 & 0 & 0 & \lambda I_n & -I_n & 0 \\
0 & 0 & \lambda P_3+P_2 & 0 & 0 & \lambda I_n & -I_n \\
0 & 0 & 0 & \lambda P_1+P_0 & 0 & 0 & \lambda I_n \\ \hline
I_n & -\lambda I_n & 0 & 0 & 0 & 0 & 0 \\
0 & I_n & -\lambda I_n & 0 & 0 & 0 & 0 \\
0 & 0 & I_n & -\lambda I_n & 0 & 0 & 0
\end{array}\right],
\]
is a skew-symmetric strong linearization for $P(\lambda)$.
Additionally, it is not difficult to show that there exists a permutation matrix $P$ such that 
\[
P^T\mathcal{L}_1(\lambda)P=
\begin{bmatrix}
\lambda P_7+P_6 & -I_n & 0 & 0 & 0 & 0 & 0 \\
I_n & 0 & -\lambda I_n & 0 & 0 & 0 & 0 \\
0 & \lambda I_n & \lambda P_5+P_4 & -I_n & 0 & 0 & 0 \\
0 & 0 & I_n & 0 & -\lambda I_n & 0 & 0 \\
0 & 0 & 0 & \lambda I_n & \lambda P_3+P_2 & -I_n & 0 \\
0 & 0 & 0 & 0 & I_n & 0 & -\lambda I_n \\
0 & 0 & 0 & 0 & 0 & \lambda I_n & \lambda P_1+P_0
\end{bmatrix},
\]
which is a slight variation of the  block tridiagonal skew-symmetric companion forms introduced in \cite{Greeks,Skew}.
\end{example}

Notice that the pencils $\lambda B+A$ in the block Kronecker pencils in Examples \ref{ex:sym_odd2} and \ref{ex:sym_odd3} are (i) block symmetric (recall Definition \ref{def:block_transpose}); and (ii) constructed using blocks of the form \eqref{eq:M_block_entry}, for some constants $\alpha_{ij},\beta_{ij}$ and natural numbers $\ell,t$.
If the matrix polynomial $P(\lambda)$ is skew-symmetric, blocks of the form \eqref{eq:M_block_entry} are also skew-symmetric.
But block symmetric pencils with skew-symmetric blocks are skew-symmetric, as we show in Lemma \ref{lemma:skew-sym}.
\begin{lemma}\label{lemma:skew-sym}
Let $\lambda B+A$ be be an $pn\times pn$ pencil partitioned into $p\times p$ blocks of size $n\times n$ and let us denote by $\lambda B_{ij}+A_{ij}$ the $(i,j)$-block-entries of $\lambda B+A$.
If $\lambda B+A$ is block symmetric, i.e., $(\lambda B+A)^{\mathcal{B}}=\lambda B+A$ and all its block entries are skew-symmetric pencils, then the pencil $\lambda B+A$ is skew-symmetric.
\end{lemma}

Lemma \ref{lemma:skew-sym} implies that any block symmetric pencil $\lambda B+A$ using blocks of the form \eqref{eq:M_block_entry} constructed to obtain symmetric strong linearizations from symmetrizable block Kronecker pencils can also be used to obtain  strong skew-symmetric linearizations from skew-symmetrizable block Kronecker pencils.
We illustrate this in the following example.
\begin{example}\label{ex:skew-sym_odd3}
Let $P(\lambda)=\sum_{k=0}^5P_k\lambda^k \in\FF[\lambda]^{n\times n}$ be a skew-symmetric matrix polynomial, and let us consider the pencils $\lambda B+A$ in the  block Kronecker pencils in Example \ref{ex:sym_odd3}, i.e., the pencils
\[
\begin{bmatrix}
\lambda P_5-P_4 & \lambda P_4 & 0 \\
\lambda P_4 & \lambda P_3+P_2 & P_1 \\
0 & P_1 & -\lambda P_1+P_0
\end{bmatrix}\, \mbox{and} \,\,
\begin{bmatrix}
\lambda P_5+P_4 & P_3 & P_2 \\
P_3 & -\lambda P_3+P_2 & -\lambda P_2+P_1\\
P_2 & -\lambda P_2+P_1 & -\lambda P_1+P_0
\end{bmatrix}.
\]
Both pencils are block symmetric with blocks of the form \eqref{eq:M_block_entry}, thus, we obtain from Lemma \ref{lemma:skew-sym} that they are skew-symmetric.
In addition, they satisfy \eqref{eq:condition_M}, which implies, by Corollary \ref{cor:givenPblockKron}, that the following  block Kronecker pencils
\[ 
\left[\begin{array}{ccc|cc}
\lambda P_5-P_4 & \lambda P_4 & 0 & -I_n & 0 \\
\lambda P_4 & \lambda P_3+P_2 & P_1 & \lambda I_n & -I_n \\
0 & P_1 & -\lambda P_1+P_0 & 0 & \lambda I_n\\ \hline
I_n & -\lambda I_n & 0 & 0 & 0 \\
0 & I_n & -\lambda I_n & 0 & 0
\end{array}\right]
\]
and
\[
\left[\begin{array}{ccc|cc}
\lambda P_5+P_4 & P_3 & P_2 & -I_n & 0 \\
P_3 & -\lambda P_3+P_2 & -\lambda P_2+P_1 & \lambda I_n & -I_n \\
P_2 & -\lambda P_2+P_1 & -\lambda P_1+P_0 & 0 & \lambda I_n \\ \hline
I_n & -\lambda I_n & 0 & 0 & 0 \\
0 & I_n & -\lambda I_n & 0 & 0
\end{array}\right]
\]
are skew-symmetric strong linearizations of $P(\lambda)$.
Finally, notice that both pencils are skew-symmetric companion forms of degree-5 skew-symmetric matrix polynomials.
\end{example}

\section{Modified symmetrizable and skew-symmetrizable block Kronecker pencils associated with even-grade matrix polynomials with nonsingular leading coefficients}\label{sec:Kronecker_pencils_even}

In this section we introduce two new  families of minimal bases pencils that we have named \emph{modified symmetrizable block Kronecker pencils} and \emph{modified skew-symmetrizable block Kronecker pencils}. These will be used to construct strong linearization of symmetric and skew-symmetric even-degree matrix polynomials with nonsingular leading or trailing coefficients.
\begin{definition}\label{def:blockKronlin_even}
Let $\widehat{L}_k(\lambda)$ be the matrix pencil defined in \eqref{eq:Lk2}, let $t$ and $n$ be nonzero natural numbers, let $\sigma\in\{1,-1\}$, and let $\lambda B+A$ be an arbitrary pencil of size $(t+1)n\times (t+1)n$.
Then, any matrix pencil of the form
\begin{equation}\label{eq:linearization_even}
\mathcal{L}(\lambda)=
\left[\begin{array}{c|c}
\lambda B+A & \widehat{L}_t(\lambda)^T\otimes I_n \\ \hline
\sigma \widehat{L}_t(\lambda)\otimes I_n & \phantom{\Big{(}} 0 \phantom{\Big{(}}
\end{array}\right]
\end{equation}
is called a \emph{modified symmetric block Kronecker pencil} if $\sigma=1$ or a \emph{modified skew-symmetric Kronecker pencil} if $\sigma=-1$.
Additionally, for simplicity or when the scalar $\sigma$ is not specified, the pencil \eqref{eq:linearization_even} is called a \emph{modified block Kronecker pencil}.
\end{definition}
\begin{remark}\label{remark:partition}
{\rm If we partition the pencil $(t+1)n\times (t+1)n$ pencil $\lambda B+A$ in \eqref{eq:linearization_even} as follows
\begin{equation}\label{eq:partition}
\lambda B+A = 
\begin{bmatrix}
M_{11}(\lambda) & M_{12}(\lambda) \\
M_{21}(\lambda) & M_{22}(\lambda)
\end{bmatrix},
\end{equation}
where $M_{11}(\lambda)\in\FF[\lambda]^{n\times n}$, $M_{12}(\lambda),M_{21}(\lambda)^T\in\FF[\lambda]^{n\times tn}$ and $M_{22}(\lambda)\in\FF[\lambda]^{tn\times tn}$, the pencil \eqref{eq:linearization_even} may be partitioned as
\[ 
\mathcal{L}(\lambda) = \left[\begin{array}{c|c|c}
M_{11}(\lambda) & M_{12}(\lambda) & 0 \\ \hline
\sigma M_{12}(\lambda)^T & \phantom{\Big{(}} M_{22}(\lambda) \phantom{\Big{(}} & L_{t-1}(\lambda)^T \otimes I_n \\ \hline
0 & L_{t-1}(\lambda)\otimes I_n & 0
\end{array}\right],
\] 
where the matrix polynomial $L_k(\lambda)$ has been defined in \eqref{eq:Lk}.
}
\end{remark}

From Example \ref{ex-L-Lamb2} we obtain that the matrix polynomials $\sigma \widehat{L}_t(\lambda)\otimes I_n$ and $\widehat{\Lambda}_t(\lambda)\otimes I_n$ and $\sigma \rev \widehat{L}_t(\lambda)\otimes I_n$ and $\widetilde{\Lambda}_t(\lambda)\otimes I_n$, with $\sigma\in\{-1,1\}$, are two pairs of dual minimal bases.
Thus, as a immediate corollary of Theorem \ref{thm:blockminlin} we  obtain the following result for modified block Kronecker pencils.
\begin{theorem}\label{thm:strong_even}
Let $\mathcal{L}(\lambda)$ be the pencil in \eqref{eq:linearization_even}, and let $\widehat{\Lambda}_k(\lambda)$ and $\widetilde{\Lambda}_k(\lambda)$ be the matrix polynomials in \eqref{eq:Lambda2} and \eqref{eq:Lambda3}, respectively.
Then $\mathcal{L}(\lambda)$ is a linearization of the matrix polynomial
\begin{equation}\label{eq:condition_even}
P(\lambda):=(\widehat{\Lambda}_t(\lambda)^T\otimes I_n)(\lambda B+A)(\widehat{\Lambda}_t(\lambda)\otimes I_n)
\end{equation}
of grade $2t$. 
In addition, the pencil $\rev \mathcal{L}(\lambda)$ is a linearization of the matrix polynomial
\begin{equation}\label{eq:condition_even_tilde}
\widetilde{P}(\lambda):=(\widetilde{\Lambda}_t(\lambda)^T\otimes I_n)(\lambda A+B)(\widetilde{\Lambda}_t(\lambda)\otimes I_n)
\end{equation}
of grade also $2t$. 
\end{theorem}
\begin{remark}{\rm
Notice that the matrix polynomials $P(\lambda)$ and $\widetilde{P}(\lambda)$ in \eqref{eq:condition_even} and \eqref{eq:condition_even_tilde} have even grade.
This is the reason why we say that the families of modified symmetrizable and skew-symmetrizable block Kronecker pencils are {\rm associated with even-grade matrix polynomials}.
}
\end{remark}

In general it is not true that $\widetilde{P}(\lambda)=\rev P(\lambda)$, where $P(\lambda)$ and $\widetilde{P}(\lambda)$ are the matrix polynomials in \eqref{eq:condition_even} and \eqref{eq:condition_even_tilde}.
Therefore, the modified block Kronecker pencil $\mathcal{L}(\lambda)$ in \eqref{eq:linearization_even} is not, in general, a strong linearization for $P(\lambda)$.
However, we show in Theorem \ref{thm:mild_conditions} that under some extra conditions we can obtain modified block Kronecker pencils that are, indeed, strong linearizations for a certain matrix polynomial.
To prove this result, the following lemma will be useful.
\begin{lemma}\label{lemma:aux_even}
Let $d$ be an even number, let $Q(\lambda)\in\FF[\lambda]^{n\times n}$ be a degree-$(d-1)$ matrix polynomial, let $P\in\FF^{n\times n}$ be a nonsingular matrix, and set $t:=d/2$.
Then, the following statements hold.
\begin{enumerate}
\item[\rm (a)] The matrix polynomials
\[
P(\lambda)=\lambda^d P + Q(\lambda) \quad \mbox{and} \quad 
\widehat{P}(\lambda)=
\begin{bmatrix}
-P & \lambda^t P \\
\lambda^t P & Q(\lambda)
\end{bmatrix}
\] 
are extended unimodularly equivalent. 
\item[\rm (b)] The matrix polynomials
\[
\rev_d P(\lambda)=P + \lambda \rev_{d-1}Q(\lambda) \quad \mbox{and} \quad 
\widetilde{P}(\lambda)=
\begin{bmatrix}
-\lambda P &  P \\
 P & \rev_{d-1} Q(\lambda)
\end{bmatrix}
\] 
are extended unimodularly equivalent.
\end{enumerate}
\end{lemma}
\begin{proof}
Since the matrix $P$ is nonsingular, the following matrix polynomials are unimodular
\[
\begin{bmatrix}
-P^{-1} & 0 \\
\lambda^t I_n & I_n
\end{bmatrix} \quad \mbox{and} \quad
\begin{bmatrix}
I_n & -\lambda^t I_n \\
0 & I_n
\end{bmatrix}.
\]
Then, notice
\[
\begin{bmatrix}
-P^{-1} & 0 \\
\lambda^t I_n & I_n
\end{bmatrix}
\begin{bmatrix}
-P & \lambda^t P \\
\lambda^t P & Q(\lambda)
\end{bmatrix}
\begin{bmatrix}
I_n & -\lambda^t I_n \\
0 & I_n
\end{bmatrix} =
\begin{bmatrix}
I_n & 0 \\ 0 & P(\lambda)
\end{bmatrix},
\]
so part (a) is true. 
In addition, the nonsingularity of $P$ also implies that the following matrix polynomials
\[
\begin{bmatrix}
P^{-1} & 0 \\
-\rev_{d-1}Q(\lambda)P^{-1} & I_n 
\end{bmatrix} \quad \mbox{and} \quad
\begin{bmatrix}
0 & I_n \\
I_n & \lambda I_n
\end{bmatrix},
\]
are unimodular. 
Finally, notice
\[
\begin{bmatrix}
P^{-1} & 0 \\
-\rev_{d-1}Q(\lambda)P^{-1} & I_n 
\end{bmatrix}
\begin{bmatrix}
-\lambda P &  P \\
 P & \rev_{d-1} Q(\lambda)
\end{bmatrix}
\begin{bmatrix}
0 & I_n \\
I_n & \lambda I_n
\end{bmatrix} =
\begin{bmatrix}
I_n & 0 \\
0 & \rev_d P(\lambda)
\end{bmatrix},
\]
so part (b) is also true.
\end{proof}

In Theorem \ref{thm:mild_conditions} we show that under some extra conditions modified  block Kronecker pencils are strong linearizations for certain matrix polynomials. 
\begin{theorem}\label{thm:mild_conditions}
Let $\mathcal{L}(\lambda)$ be the pencil in \eqref{eq:linearization_even}, and let $\widehat{\Lambda}_k(\lambda)$ be the matrix polynomial in \eqref{eq:Lambda2}.
If
\begin{equation}\label{eq:condition_new}
(\widehat{\Lambda}_t(\lambda)^T\otimes I_n)(\lambda B+A)(\widehat{\Lambda}_t(\lambda)\otimes I_n) = 
\begin{bmatrix}
-P & \lambda^t P \\
\lambda^t P & Q(\lambda)
\end{bmatrix},
\end{equation}
for some nonsingular matrix $P\in\mathbb{F}^{n\times n}$ and some grade-$(2t-1)$ matrix polynomial $Q(\lambda)\in \FF[\lambda]^{n\times n}$, then the pencil $\mathcal{L}(\lambda)$ is a strong linearization of the grade-$2t$ matrix polynomial $P(\lambda):=\lambda^{2t} P+Q(\lambda)$.
\end{theorem}
\begin{proof}
That the pencil $\mathcal{L}(\lambda)$ is a linearization of $P(\lambda)$ follows directly from Theorem \ref{thm:strong_even} and Lemma \ref{lemma:aux_even}(a), so let us prove that $\rev \mathcal{L}(\lambda)$ is a linearization of $\rev P(\lambda)$.
To this aim, let us partition the $(t+1)n\times (t+1)n$ pencil $\lambda B+A$ as in \eqref{eq:partition}.
Then, the equation \eqref{eq:condition_new} implies:
\begin{itemize}
\item[\rm (i)] $M_{11}(\lambda) = -P$;
\item[\rm (ii)] $M_{12}(\lambda)(\Lambda_{t-1}(\lambda)\otimes I_n)  = \lambda^t P$,
\item[\rm (iii)]$(\Lambda_{t-1}(\lambda)^T\otimes I_n)M_{21}(\lambda)=\lambda^t P$, and
\item[\rm (iv)] $(\Lambda_{t-1}(\lambda)^T\otimes I_n)M_{22}(\lambda)(\Lambda_{t-1}(\lambda)\otimes I_n) = Q(\lambda)$.
\end{itemize}
Computing the 1-reversal, $t-$reversal, $t-$reversal and $(d-1)$-reversal, respectively, of both sides of the equations  (i), (ii), (iii) and (iv), it may be checked that 
\[
(\widetilde{\Lambda}_t(\lambda)^T\otimes I_n)(\lambda A+B)(\widetilde{\Lambda}_t(\lambda)\otimes I_n) = 
\begin{bmatrix}
-\lambda P & P \\
P & \rev_{d-1} Q(\lambda)
\end{bmatrix}
\]
holds.
This fact, together with part-(b) in Lemma \ref{lemma:skew-sym}, implies that the pencil $\rev \mathcal{L}(\lambda)$ is a linearization of $\rev P(\lambda)$.
Therefore, $\mathcal{L}(\lambda)$ is a strong linearization of $P(\lambda)$.
\end{proof}

\begin{remark}{\rm
Notice that the polynomial $P(\lambda)$ in Theorem \ref{thm:mild_conditions} has a nonsingular  leading coefficient $P$.
This is the reason why the pencils introduced in this section are said to be \emph{associated with matrix polynomials with nonsingular leading coefficients}.
}
\end{remark}

The next step is to show how to construct strong linearizations for a fixed matrix polynomial $P(\lambda)=\sum_{k=0}^d P_k\lambda^k \in\FF[\lambda]^{n\times n}$ with even degree and nonsingular leading coefficient from modified  block Kronecker pencils.
To this aim, we will write $P(\lambda)=\lambda^P P_d+Q(\lambda)$, where \begin{equation}\label{eq:Q}
Q(\lambda):=P(\lambda)-\lambda^dP_d = \lambda^{d-1}P_{d-1} + \cdots + \lambda P_1 +P_0.
\end{equation}
\begin{remark}\label{remark:QandP}
{\rm 
Notice that the matrix polynomials $Q(\lambda)$ and $P(\lambda)$ share the same structure that $P(\lambda)$ might posses, i.e., if $P(\lambda)$ is a symmetric or skew-symmetric matrix polynomial, then the matrix polynomial $Q(\lambda) = P(\lambda)-\lambda^dP_d$ is, respectively, symmetric or skew-symmetric.
} 
\end{remark}
From Theorem \ref{thm:mild_conditions} we obtain that the modified  block Kronecker pencil $\mathcal{L}(\lambda)$ in \eqref{eq:linearization_even} is a strong linearization of $P(\lambda)$ if  \eqref{eq:condition_new} holds with $Q(\lambda)$ as in \eqref{eq:Q} and $P=P_d$.
From this, and after some simple algebraic manipulations, we obtain in Corollary \ref{cor:givenPblockKron_even}  two equivalent conditions on the pencil $\lambda B+A$ that guarantee that  $\mathcal{L}(\lambda)$ is a strong linearization of $P(\lambda)$.
\begin{corollary} \label{cor:givenPblockKron_even}
Let $P(\lambda) = \sum_{k=0}^d P_k \lambda^k \in \FF[\lambda]^{n \times n}$ with even degree $d$, let $Q(\lambda) = \sum_{k=0}^{d-1} P_k \lambda^k \in \FF[\lambda]^{n \times n}$, let $\Lambda_{k}(\lambda)$, $\widehat{\Lambda}_{k}(\lambda)$ and $\Gamma_{k}(\lambda)$ be the matrix polynomials in \eqref{eq:Lambda}, \eqref{eq:Lambda2} and \eqref{eq:Gamma}, respectively, let $\mathcal{L}(\lambda)$ be the pencil in \eqref{eq:linearization_even} with $t=d/2$, and let us consider the pencil $\lambda B+A$ partitioned as in \eqref{eq:partition}.
Then, any of the following conditions guarantees that the pencil $\mathcal{L}(\lambda)$ is a strong linearization of $P(\lambda)$:
\begin{enumerate}
\item[\rm (a)] 
\begin{align*}
\begin{split}
&P_d=-M_{11}(\lambda), \\
&\lambda^t P_d = M_{12}(\lambda)(\Lambda_{t-1}(\lambda)\otimes I_n) = (\Lambda_{t-1}(\lambda)^T\otimes I_n)M_{21}(\lambda),\\
&Q(\lambda) = (\Lambda_{t-1}(\lambda)^T\otimes I_n)M_{22}(\lambda)(\Lambda_{t-1}(\lambda)\otimes I_n).
\end{split}
\end{align*}
\item[\rm (b)]
\begin{align}\label{eq:condition_M_even}
\begin{split}
P_d=&-M_{11}(\lambda), \\
\lambda^t P_d=&\su\left( M_{12}(\lambda)\odot \Lambda_{t-1}(\lambda)\otimes I_n \right)=\\
&\su\left( M_{21}(\lambda)\odot \Lambda_{t-1}(\lambda)^T\otimes I_n \right),  \\
Q(\lambda)=&\su\left( M_{22}(\lambda) \odot \Gamma_{t-1}(\lambda) \right).
\end{split}
\end{align}
\end{enumerate}
\end{corollary}

The equations in \eqref{eq:condition_M_even} allow us to easily check if a modified  block Kronecker pencil is a strong linearization of $P(\lambda)$.
We illustrate this in the following example, where we show that  certain modified  block Kronecker pencils are strong linearizations for degree-6 matrix polynomials with nonsingular leading coefficients.
\begin{example}
Let $P(\lambda)=\sum_{k=0}^6 P_k\lambda^k\in\FF[\lambda]^{n\times n}$ with nonsingular leading coefficient, let $Q(\lambda)=\sum_{k=0}^5 P_k\lambda^k\in\FF[\lambda]^{n\times n}$ and let $\sigma\in\{-1,1\}$.
We are going to show that the modified  block Kronecker pencil
\[
\left[\begin{array}{cccc|cc}
-P_6 & \lambda P_6-P_5 & \lambda P_5 & 0 & 0 & 0 \\
0 & \lambda P_5 & \lambda P_4 & 0 & -I_n & 0 \\
P_4 & P_3 & 0 & P_1 & \lambda I_n & -I_n \\
-\lambda P_4 & 0 & \lambda P_2 & P_0 & 0 & \lambda I_n \\ \hline
0 & -\sigma I_n & \sigma \lambda I_n & 0 & 0 & 0 \\
0 & 0 & -\sigma I_n & \sigma \lambda I_n & 0 & 0
\end{array}\right]
\]
is a strong linearization of $P(\lambda)$.
Indeed, this follows from Corollary \ref{cor:givenPblockKron_even}(b) together with the following equalities
\begin{align*}
&\su\left( \begin{bmatrix}
\lambda P_6-P_5 & \lambda P_5 & 0 
\end{bmatrix} \odot \begin{bmatrix}
\lambda^2 I_n & \lambda I_n & I_n 
\end{bmatrix}\right) = \\
&\su\left( \begin{bmatrix}
\lambda^3 P_6-\lambda^2 P_5 & \lambda^2 P_5 & 0 
\end{bmatrix} \right) = \\
&\su\left( \begin{bmatrix}
\lambda P_6 \\ P_4 \\ -\lambda P_4 
\end{bmatrix} \odot \begin{bmatrix}
\lambda^2 I_n \\ \lambda I_n \\ I_n 
\end{bmatrix}\right) = 
\su\left( \begin{bmatrix}
\lambda^3 P_6 \\ \lambda  P_4 \\ -\lambda P_4 
\end{bmatrix} \right) = \lambda^3 P_6, 
\end{align*}
and
\begin{align*}
&\su\left( \begin{bmatrix} \lambda P_5 & \lambda P_4  & 0 \\
P_3 & 0 & P_1 \\
0 & \lambda P_2 & P_0 \end{bmatrix} \odot \begin{bmatrix}
\lambda^4 I_n & \lambda^3 I_n & \lambda^2 I_n \\
\lambda^3 I_n & \lambda^2 I_n & \lambda I_n \\
\lambda^2 I_n & \lambda I_n & I_n
\end{bmatrix} \right) = \\
&\su\left( \begin{bmatrix} \lambda^5 P_5 & \lambda^4 P_4  & 0 \\
\lambda^3 P_3 & 0 & \lambda P_1 \\
0 & \lambda^2 P_2 & P_0 \end{bmatrix} \right) = Q(\lambda),
\end{align*}
which are trivial to check.
\end{example}

So far, we showed  that strong linearizations  for an even-degree matrix polynomial $P(\lambda)$ can be obtained from modified block Kronecker pencils.
In the following subsection, we present a methodology for obtaining from this family of pencils linearizations that are symmetric or skew-symmetric whenever $P(\lambda)$ is, respectively, symmetric or skew-symmetric.

\subsection{Structure-preserving linearizations for even-degree symmetric or skew-symmetric matrix polynomials}\label{sec:structured-lin:even}

Let $P(\lambda)=\sum_{k=0}^dP_k\lambda^k\in\FF[\lambda]^{n\times n}$ be a symmetric or skew-symmetric even-degree matrix polynomial.
In this section, we show how to construct structure-preserving strong linearizations of $P(\lambda)$ when its leading and/or trailing coefficients (i.e., $P_d$ or $P_0$) are nonsingular.
We focus mainly on the case $P_d$ nonsingular. 
The case $P_0$ nonsingular is considered at the end as a corollary.

Modified  block Kronecker pencils \eqref{eq:linearization_even} are symmetric or skew-symmetric if and only if their pencils $\lambda B+A$ are, respectively, symmetric or skew-symmetric. 
Additionally, a modified block Kronecker pencil needs to satisfy the conditions on Corollary \ref{cor:givenPblockKron_even} in order to be a strong linearization for $P(\lambda)$.
For this reason,  our goal, now, is to obtain all the pencil
\begin{equation}\label{eq:linearization_even2}
\mathcal{L}(\lambda) = \left[\begin{array}{c|c|c}
-P_d & M_{12}(\lambda) & 0 \\ \hline
\sigma M_{12}(\lambda)^T & \phantom{\Big{(}} M_{22}(\lambda) \phantom{\Big{(}} & L_{t-1}(\lambda)^T \otimes I_n \\ \hline
0 & \sigma L_{t-1}(\lambda)\otimes I_n & 0
\end{array}\right], \quad \mbox{with }t=\frac{d}{2},
\end{equation}
where $\sigma=1$ corresponds to the symmetric case and $\sigma=-1$ to the skew-symmetric case, and where  $M_{12}(\lambda)\in\FF[\lambda]^{n\times tn}$ and $M_{22}(\lambda)\in\FF[\lambda]^{tn\times tn}$, satisfying the conditions
\[
\begin{split}
&\lambda^t P_d=\su\left( M_{12}(\lambda)\odot \Lambda_{t-1}(\lambda)\otimes I_n \right) \quad \mbox{and}\\
&Q(\lambda)=\su\left( M_{22}(\lambda) \odot \Gamma_{t-1}(\lambda) \right),
\end{split}
\]
or, equivalently,
\begin{align}
\label{eq:first_condition}&\lambda^t P_d = M_{12}(\lambda)(\Lambda_{t-1}(\lambda)\otimes I_n)  \quad \mbox{and}\\
\label{eq:second_condition}&Q(\lambda) = (\Lambda_{t-1}(\lambda)^T\otimes I_n)M_{22}(\lambda)(\Lambda_{t-1}(\lambda)\otimes I_n),
\end{align}
where the matrix polynomials $Q(\lambda)$, $\Lambda_k(\lambda)$ and $\Gamma_k(\lambda)$ have been defined in \eqref{eq:Q}, \eqref{eq:Lambda} and \eqref{eq:Gamma}, respectively.

Our first step, then, is Proposition \ref{proposition:solutionM12}, where we obtain all the pencil solutions $M_{12}(\lambda)$ of the equation \eqref{eq:first_condition}.
Since the proof of Proposition \ref{proposition:solutionM12} is very simple, it is omitted.
\begin{prop}\label{proposition:solutionM12}
Let $\Lambda_k(\lambda)$ be the matrix polynomial defined in \eqref{eq:Lambda}, let $P(\lambda)=\sum_{k=0}^dP_k\lambda^k\in\FF[\lambda]^{n\times n}$ be an even-degree matrix polynomial, set $t=d/2$, and let $M_{12}(\lambda)$ be a matrix pencil with size $n\times tn$.
Then, the solution $M_{12}(\lambda)$ of \eqref{eq:first_condition} is give by
\[
\begin{bmatrix}
\lambda P_d+W_1 & -\lambda W_1 + W_2 & -\lambda W_2 + W_3 & \cdots & \lambda W_{t-1}+W_t & -\lambda W_t
\end{bmatrix},
\]
where $W_1,\hdots,W_t$ are arbitrary $n\times n$ matrices.
\end{prop}
Our second step is to solve \eqref{eq:second_condition} with a symmetric or skew-symmetric pencil  $M_{22}(\lambda)$, depending on whether the matrix polynomial $P(\lambda)$ is symmetric or skew-symmetric. 
Looking closely at this equation, we see that it is just \eqref{eq:condition} with $s=t-1$ and with $Q(\lambda)$ instead of $P(\lambda)$ on the left-hand-side. 
Since, according to Remark \ref{remark:QandP}, the polynomials $P(\lambda)$ and $Q(\lambda)$ share the same structure, the symmetric and skew-symmetric solutions of \eqref{eq:second_condition} may be obtained from Theorems \ref{thm:solving_sym} and \ref{thm:solving_skew-sym} with $Q(\lambda)$ instead of $P(\lambda)$.

\begin{remark}\label{remark:procedure}
{\rm The previous considerations implies the following procedure to construct a structure-preserving strong linearization of a symmetric or skew-symmetric even-degree matrix polynomial $P(\lambda)=\sum_{k=0}^d P_k\lambda^k\in\FF[\lambda]^{n\times n}$ with nonsingular leading coefficient. 
First, solve \eqref{eq:first_condition}, whose general solution is in Proposition \ref{proposition:solutionM12}, and, then, solve \eqref{eq:second_condition} with a symmetric or skew-symmetric matrix pencil $M_{22}(\lambda)$, depending on whether $P(\lambda)$ is symmetric or skew-symmetric.  
The solution of the latter equation may be obtained from Theorems \ref{thm:solving_sym} and \ref{thm:solving_skew-sym}.
Then, the pencil $\mathcal{L}(\lambda)$ in \eqref{eq:linearization_even2} is a structure-preserving strong linearization for $P(\lambda)$. }
\end{remark}

We illustrate the procedure outline in Remark \ref{remark:procedure} in the following example, where we obtain a structure-preserving linearization from a modified block Kronecker pencil that can be permuted into a block-tridiagonal pencil.
\begin{example}\label{ex:sym_even2}
Let $P(\lambda)=\sum_{k=0}^8P_k\lambda^k\in\FF[\lambda]^{n\times n}$ be a symmetric or skew-symmetric matrix polynomial with nonsingular leading coefficient.
The most simple solution of $\lambda^4 P_8 = M_{12}(\lambda)(\Lambda_{3}(\lambda)\otimes I_n)$ is $M_{12}(\lambda) =
\begin{bmatrix}
\lambda P_8 & 0 & \cdots & 0
\end{bmatrix}$.
Additionally, as we have seen in Examples \ref{ex:sym_odd2} and \ref{ex:skew-sym_odd2}, the most simple symmetric or skew-symmetric solution to $P(\lambda)-\lambda^8 P_8 = (\Lambda_{3}(\lambda)^T\otimes I_n)M_{22}(\lambda)(\Lambda_{3}(\lambda)\otimes I_n)$ is the block diagonal pencil 
$M_{22}(\lambda) = \diag(\lambda P_7+P_6,\lambda P_5+P_4,\lambda P_3+P_2,\lambda P_1+P_0)$.
Therefore, we get that the modified block Kronecker pencil $\mathcal{L}(\lambda)$ given by
\[
\left[\begin{array}{ccccc|ccc}
-P_8 & \lambda P_8 & 0 & 0 & 0 & 0 & 0 & 0 \\
\lambda P_8 & \lambda P_7+P_6 & 0 & 0 & 0 & -I_n & 0 & 0 \\
0 & 0 & \lambda P_5+P_4 & 0 & 0 & \lambda I_n & -I_n & 0 \\
0 & 0 & 0 & \lambda P_3+P_2 & 0 & 0 & \lambda I_n & -I_n \\
0 & 0 & 0 & 0 & \lambda P_1+P_0 & 0 & 0 & \lambda I_n \\ \hline
0 & -\sigma I_n & \sigma \lambda I_n & 0 & 0 & 0 & 0 & 0 \\
0 & 0 &  -\sigma I_n & \sigma \lambda I_n & 0 & 0 & 0 & 0 \\
0 & 0 & 0 & -\sigma I_n & \sigma \lambda I_n & 0 & 0 & 0
\end{array}\right]
\]
is a  strong linearization  that for $\sigma=1$ is symmetric whenever $P(\lambda)$ is symmetric, and for $\sigma=-1$ is skew-symmetric whenever $P(\lambda)$ is skew-symmetric.
Additionally, it is not difficult to check that there exists a permutation matrix $P$ such that $P^T\mathcal{L}(\lambda)P$ is equal to
\[
\begin{bmatrix}
-P_8 & \lambda P_8 & 0 & 0 & 0 & 0 & 0 & 0 \\
\lambda P_8 & \lambda P_7+P_6 & -I_n & 0 & 0 & 0 & 0 & 0 \\
0 & -\sigma I_n & 0 & \sigma \lambda I_n & 0 & 0 & 0 & 0 \\
0 & 0 & \lambda I_n & \lambda P_5+P_4 & -I_n & 0 & 0 & 0 \\
0 & 0 & 0 & -\sigma I_n & 0 & \sigma \lambda I_n & 0 & 0 \\
0 & 0 & 0 & 0 & \lambda I_n & \lambda P_3+P_2 & -I_n & 0 \\
0 & 0 & 0 & 0 & 0 & -\sigma I_n & 0 & \sigma \lambda I_n\\
0 & 0 & 0 & 0 & 0 & 0 & \lambda I_n & \lambda P_1+P_0
\end{bmatrix},
\]
which is a block-tridiagonal symmetric (if $\sigma=1$) or skew-symmetric (if $\sigma=-1$) companion form for, respectively, symmetric or skew-symmetric even-degree matrix polynomials with nonsingular leading coefficients.
These companion forms can be easily generalized for any even-degree matrix polynomials with nonsingular leading coefficient.
\end{example}

In the following example, we show that the symmetric companion forms of matrix polynomials with degree 4 and nonsingular leading coefficients in \cite[Example 5.8]{FPR1}, are, indeed, (up to a permutation) modified Kronecker pencils. 
\begin{example}
Let $P(\lambda)=\sum_{k=0}^4P_k\lambda^k\in\FF[\lambda]^{n\times n}$ be a symmetric matrix polynomial with nonsingular leading coefficient.
Let us consider the symmetric companion form $L_5(\lambda)$ in {\rm \cite[Example 5.8]{FPR1}}, i.e., the pencil
\[
L_5(\lambda) = \begin{bmatrix}
0 & 0 & -I_n & \lambda I_n \\
0 & -P_4 & \lambda P_4-P_3 & \lambda P_3 \\
-I_n & \lambda P_4-P_3 & \lambda P_3-P_2 & \lambda P_2 \\
\lambda I_n & \lambda P_3 & \lambda P_2 & \lambda P_1+P_0
\end{bmatrix}.
\]
Then, there exists a permutation matrix $P$ such that
\[
P^TL_5(\lambda)P = 
\left[\begin{array}{ccc|c}
-P_4 & \lambda P_4-P_3 & \lambda P_3 & 0 \\
\lambda P_4-P_3 & \lambda P_3-P_2 & \lambda P_2 & -I_n \\
\lambda P_3 & \lambda P_2 & \lambda P_1+P_0 & \lambda I_n \\ \hline
0 & -I_n & \lambda I_n & 0,
\end{array}\right],
\]
which is a modified symmetric block Kronecker pencil \eqref{eq:linearization_even2} with
\[
M_{12}(\lambda)=\begin{bmatrix}
\lambda P_4-P_3 & \lambda P_3 
\end{bmatrix} \quad \mbox{and} \quad
M_{22}(\lambda)=\begin{bmatrix}
\lambda P_3-P_2 & \lambda P_2 \\
\lambda P_2 & \lambda P_1+P_0 
\end{bmatrix}.
\]
Let us also consider the symmetric companion form $L_9(\lambda)$ in {\rm \cite[Example 5.8]{FPR1}}, i.e., the pencil
\[
L_7(\lambda)=
\begin{bmatrix}
-P_4 & 0 & \lambda P_4 & 0 \\
0 & 0 & -I_n & \lambda I_n \\
\lambda P_4 & I_n & \lambda P_3+P_2 & \lambda P_2 \\
0 & \lambda I_n & \lambda P_2 & \lambda P_1+P_0
\end{bmatrix},
\]
which is also obtained from a modified  block Kronecker pencil permuting some of its block rows and columns, i.e., there exists a permutation matrix $Q$ such that
\[
Q^TL_9(\lambda)Q=\left[\begin{array}{ccc|c}
-P_4 & \lambda P_4 & 0 & 0 \\
\lambda P_4 & \lambda P_3-P_2 & \lambda P_2 & -I_n \\
0 & \lambda P_2 & \lambda P_1+P_0 & \lambda I_n \\ \hline
0 & -I_n & \lambda I_n & 0
\end{array}\right].
\]
It is not difficult to show that the symmetric companion form $L_7(\lambda)$ in {\rm \cite[Example 5.8]{FPR1}} is also a permuted modified block Kronecker pencil.
\end{example} 

We finally consider the problem of constructing symmetric or skew-symmetric strong linearizations for $P(\lambda)$ when its trailing coefficient is nonsingular. 
The key tool is the following lemma, which is a particular case of \cite[Corollary 8.6]{Mobius} where the authors study the interaction between linearizations and M{\"o}bius transformations of matrix polynomials \footnote{the $\rev(\cdot)$ operation is a particular case of a M{\"o}bius transformation of a matrix polynomial \cite{Mobius}.}.
\begin{lemma}\label{lemma:sym_even}
Let $P(\lambda)$ be any  $n\times n$ matrix polynomial, and define $\widetilde{P}(\lambda):=\rev P(\lambda)$. 
Then, if $\widetilde{L}(\lambda)$ is any  strong linearization of $\widetilde{P}(\lambda)$, then $L(\lambda):=\rev \widetilde{L}(\lambda)$ is a strong linearization of $P(\lambda)$. 
\end{lemma}

If a matrix polynomial $P(\lambda)$  with nonsingular trailing coefficient is  symmetric (skew-symmetric), then  $\rev P(\lambda)$ is a symmetric  (skew-symmetric) matrix polynomial with nonsingular leading coefficient. 
Thus, from any structure-preserving strong linearization for symmetric (skew-symmetric) matrix polynomials of even degree with nonsingular leading coefficients and Lemma \ref{lemma:sym_even}, we can obtain structure-preserving strong linearizations for symmetric (skew-symmetric) matrix polynomials of even degree with nonsingular trailing coefficients.
 We illustrate this in the following example.
 \begin{example}
 Let $P(\lambda)=\sum_{k=0}^8P_k\lambda^k\in\FF[\lambda]^{n\times n}$ be a symmetric or skew-symmetric matrix polynomial with nonsingular trailing coefficient.
 Then, $\rev P(\lambda) = \sum_{k=0}^8 P_{8-k}\lambda^k$ is, respectively, a symmetric or skew-symmetric matrix polynomial with nonsingular leading coefficient.
 Let us consider the strong linearizations in Example \ref{ex:sym_even2}, but for the matrix polynomial $\rev P(\lambda)$ instead of $P(\lambda)$.
Then, from Lemma \ref{lemma:sym_even}, we obtain that the pencils
\[
\left[\begin{array}{ccccc|ccc}
-\lambda P_0 & P_0 & 0 & 0 & 0 & 0 & 0 & 0 \\
P_0 & \lambda P_2+P_1 & 0 & 0 & 0 & -\lambda I_n & 0 & 0 \\
0 & 0 & \lambda P_4+P_3 & 0 & 0 & I_n & -\lambda I_n & 0 \\
0 & 0 & 0 & \lambda P_6+P_5 & 0 & 0 & I_n & -\lambda I_n \\
0 & 0 & 0 & 0 & \lambda P_8+P_7 & 0 & 0 & I_n \\ \hline 
0 & -\sigma \lambda I_n & \sigma I_n & 0 & 0 & 0 & 0 & 0 \\
0 & 0 & -\sigma \lambda I_n & \sigma I_n & 0 & 0 & 0 & 0 \\
0 & 0 & 0 & -\sigma \lambda I_n & \sigma I_n & 0 & 0 & 0
\end{array}\right]
\]
and
\[
\begin{bmatrix}
-\lambda P_0 & P_0 & 0 & 0 & 0 & 0 & 0 & 0 \\
P_0 & \lambda P_2+P_1 & -\lambda I_n & 0 & 0 & 0 & 0 & 0 \\
0 & -\sigma \lambda I_n & 0 & \sigma I_n & 0 & 0 & 0 & 0 \\
0 & 0 & I_n & \lambda P_4+P_3 & -\lambda I_n & 0 & 0 & 0 \\
0 & 0 & 0 & -\sigma \lambda I_n & 0 & \sigma I_n & 0 & 0 \\
0 & 0 & 0 & 0 & I_n & \lambda P_6+P_5 & -\lambda I_n & 0 \\
0 & 0 & 0 & 0 & 0 & -\sigma \lambda I_n & 0 & \sigma I_n\\
0 & 0 & 0 & 0 & 0 & 0 & I_n & \lambda P_8+P_7
\end{bmatrix}
\]
are symmetric ($\sigma=1$) or skew-symmetric ($\sigma=-1$) companion forms for, respectively, symmetric or skew-symmetric degree-8 matrix polynomials with nonsingular trailing coefficients.
 \end{example}

In the following example, we show that the symmetric companion forms of matrix polynomials with degree 4 and nonsingular trailing coefficients in \cite[Example 5.8]{FPR1}, are, indeed, the reversals of modified block Kronecker pencils of the reversal of the polynomial (up to some row and column sign changes and permutations). 
\begin{example}
Let $P(\lambda)=\sum_{k=0}^4P_k\lambda^k\in\FF[\lambda]^{n\times n}$ be a symmetric matrix polynomial with nonsingular trailing coefficient.
Let us consider the symmetric companion form $L_6(\lambda)$ in {\rm \cite[Example 5.8]{FPR1}}, i.e., the pencil
\[
L_6(\lambda)=
\begin{bmatrix}
0 & -I_n & \lambda I_n & 0 \\
-I_n & \lambda P_4-P_3 & \lambda P_3 & 0 \\
\lambda I_n & \lambda P_3 & \lambda P_2-P_1 & P_0 \\
0 & 0 & P_0 & -\lambda P_0
\end{bmatrix}.
\]
Changing the sign of the first block row and first block column and reversing the order of the block rows and block columns, we obtain the pencil
\[
\left[\begin{array}{ccc|c}
-\lambda P_0 & P_0 & 0 & 0 \\
P_0 & \lambda P_2-P_1 & \lambda P_3 & -\lambda I_n \\
0 & \lambda P_3 & \lambda P_4-P_3 & I_n \\ \hline
0 & -\lambda I_n & I_n & 0
\end{array}\right],
\]
which can be obtained applying Lemma \ref{lemma:sym_even} to the following modified  block Kronecker pencil
\[
\left[\begin{array}{ccc|c}
-P_0 & \lambda P_0 & 0 & 0 \\
P_0 & -\lambda P_1+P_2 & P_3 & -\lambda I_n \\
0 & P_3 & -\lambda P_3+P_4 & I_n \\ \hline
0 & -I_n & \lambda I_n & 0 
\end{array}\right],
\]
which is a strong linearization of $\rev P(\lambda)$.
It is not difficult to show that also the pencils $L_8(\lambda)$ and $L_{10}(\lambda)$ in {\rm \cite[Example 5.8]{FPR1}} are (up to some row and column permutation and sign changes)  obtained applying Lemma \ref{lemma:sym_even} to  modified block Kronecker pencils.
\end{example}

\section{Structure preserving linearizations for Hermitian or skew-Hermitian matrix polynomials}\label{sec:Hermitian}

For matrix polynomials over the field $\FF=\mathbb{C}$ one may consider Hermitian or skew-Hermitian matrix polynomials, i.e., matrix polynomials $P(\lambda)=\sum_{k=0}^d P_k\lambda^k\in\mathbb{C}[\lambda]^{n\times n}$ satisfying $P_i^*=\sigma P_i$ for $i=0,1,\hdots,d$, with $\sigma\in\{-1,1\}$ and where $*$ denotes conjugate transpose \cite{GoodVibrations}.
The problem of constructing  pencils
\[
\lambda \mathcal{L}(\lambda)=\lambda \mathcal{L}_1+\mathcal{L}_2 \quad \mbox{with} \quad \mathcal{L}_1^*=\sigma \mathcal{L}_1 \,\, \mbox{and}\,\, \mathcal{L}_0^*=\sigma \mathcal{L}_0
\]
that are Hermitian (skew-Hermitian) strong linearizations for a Hermitian (skew-Hermitian) matrix polynomial is also of interest 
and has been also considered in \cite{Hermitian1,Hermitian2,symmetric,SignChar}.
Fortunately, everything that we have done  in this paper for symmetric and skew-symmetric matrix polynomials works equally well for Hermitian and skew-Hermitian matrix polynomials just replacing the transpose operation $(\cdot)^T$ by the conjugate transpose operation $(\cdot)^*$, except the ones in $L_s(\lambda)^T\otimes I_n$ and $\widehat{L}_t(\lambda)^T\otimes I_n$ in the block Kronecker pencils in Definitions \ref{def:blockKronlin} and \ref{def:blockKronlin_even}.

\section{Eigenvectors, and minimal indices and bases recovery procedures}\label{sec:recovery}

In this section, we show how to recover the eigenvectors of a regular matrix polynomial from those of any of the linearizations constructed in Sections \ref{sec:Kronecker_pencils_odd} and \ref{sec:Kronecker_pencils_even}.
In addition, for singular matrix polynomials of odd degree, we also show how to recover  minimal indices and bases.
We will only focus on right eigenvectors (resp. right minimal bases and indices), since the sets of left and right eigenvectors (resp. the sets of right and left minimal bases and the sets of right and left minimal indices) of a regular (resp. singular) symmetric or skew-symmetric matrix polynomial coincide (see \cite[Theorem 7.7]{IndexSum} and \cite[Section 3]{singular}, for example).
Throughout Sections \ref{sec:recovery_odd} and \ref{sec:recovery_even}, we will partition any  column vector  of size $nd$ into $d$ column vectors $z_1,\hdots,z_d$  of size $n$.

\subsection{Recovery procedures for matrix polynomials with odd degree}\label{sec:recovery_odd}

In Theorem \ref{thm:recover_eig_odd} we show how to recover the eigenvectors associated with finite and infinite eigenvalues of an odd-degree regular matrix polynomial  from any of its strong linearizations obtained from block Kronecker pencils.
\begin{theorem}\label{thm:recover_eig_odd}
Let $P(\lambda)=\sum_{k=0}^d P_k\lambda^k \in\FF[\lambda]^{n\times n}$ be regular odd-degree matrix polynomial. Let $\mathcal{L}(\lambda)$ be a strong linearization of $P(\lambda)$ obtained from a block Kronecker pencil \eqref{eq:linearization_general} with $s=(d-1)/2$.
Then the following statements hold.
\begin{enumerate}
\item[{\rm (a)}] If $z\in\mathbb{F}^{nd\times 1}$ is a right eigenvector of $\mathcal{L}(\lambda)$ with finite eigenvalue $\lambda_0$, then the $(s+1)$th block of $z$ is a right eigenvector of $P(\lambda)$ with finite eigenvalue $\lambda_0$.
\item[{\rm (b)}] If $z\in\mathbb{F}^{nd\times 1}$ is a right eigenvector of $\mathcal{L}(\lambda)$ for the eigenvalue $\infty$, then the first block of $z$ is a right eigenvector of $P(\lambda)$ for the eigenvalue $\infty$.
\end{enumerate}
\end{theorem}
\begin{proof}
The results follow immediately from Lemma \ref{lemma:aux}(a).
If $z$ is a right eigenvector of $\mathcal{L}(\lambda)$ with finite eigenvalue $\lambda_0$, then it is of the form
\[
z = \begin{bmatrix}
\Lambda_s(\lambda_0)\otimes I_n \\ * 
\end{bmatrix}x,
\]
for some right eigenvector $x$ of $P(\lambda)$ with eigenvalue $\lambda_0$, and  where by ``$*$'' we denote some matrix polynomial on $\lambda_0$ that is not relevant in the argument.
Then, part (a) follows from the fact that the $(s+1)$th block of $\Lambda_s(\lambda_0)\otimes I_n$ is just the identity matrix $I_n$.

In order to prove part (b), recall  that the eigenvectors of the pencil $\mathcal{L}(\lambda)$ (resp. $P(\lambda)$) corresponding to the eigenvalue $\infty$ are those of $\rev \mathcal{L}(\lambda)$ (resp. $\rev P(\lambda)$) corresponding to the eigenvalue $0$. 
As a consequence of Theorem \ref{thm:blockminlin}, the pencil $\rev \mathcal{L}(\lambda)$ is a strong minimal bases pencil (with $N_1(\lambda)=N_2(\lambda)=\rev \Lambda_s(\lambda)^T\otimes I_n$) that is a strong linearization of $\rev P(\lambda)$.
Therefore, if $z$ is a right eigenvector of $\rev\mathcal{L}(\lambda)$ with eigenvalue $\lambda_0=0$, then, applying again Lemma \ref{lemma:aux}(a) to $\rev\mathcal{L}(\lambda)$, it is of the form
\[
z = \begin{bmatrix}
\rev \Lambda_s(0)\otimes I_n \\ * 
\end{bmatrix}x,
\]
for some right eigenvector $x$ of $\rev P(\lambda)$ with eigenvalue $\lambda_0=0$, and  where by ``$*$'' we denote some matrices that are not relevant in the argument.
Then, part (b) follows from the fact that the first block of $\rev \Lambda_s(0)\otimes I_n$ is just the identity matrix $I_n$.
\end{proof}

The recovery of the minimal indices of a singular matrix polynomial $P(\lambda)$ from those of a strong linearization is, in general, a nontrivial task.
Theorem \ref{thm:recover_singular} shows how to recover minimal bases and indices of a symmetric or a skew-symmetric matrix polynomial from any of its structure-preserving linearizations obtained from  block Kronecker pencils.
\begin{theorem}\label{thm:recover_singular} 
Let $P(\lambda)=\sum_{k=0}^d P_k\lambda^k \in\FF[\lambda]^{n\times n}$ be a singular odd-degree matrix polynomial. Let $\mathcal{L}(\lambda)$ be a strong linearization of $P(\lambda)$ obtained from a block Kronecker pencil \eqref{eq:linearization_general} with $s=(d-1)/2$.
Then the following statements hold.
\begin{enumerate}
\item[{\rm (a)}] If $\{z_1(\lambda),z_2(\lambda), \hdots, z_p(\lambda)\}$ is any right minimal basis of $\mathcal{L}(\lambda)$ and if $x_j(\lambda)$ is the $(s+1)th$ block of $z_j(\lambda)$, for $j=1,2,\hdots,p$, then $\{x_1(\lambda),x_2(\lambda),\hdots, \allowbreak x_p(\lambda)\}$ is a right minimal basis of $P(\lambda)$.
\item[\rm (b)] If $0\leq \epsilon_1\leq \epsilon_2 \leq \cdots \leq \epsilon_p$ are the right minimal indices of $P(\lambda)$, then
\[
0\leq \epsilon_1+s \leq \epsilon_2+s\leq \cdots \leq \epsilon_p+s
\]
are the right minimal indices of $\mathcal{L}(\lambda)$.
\end{enumerate}
\begin{proof}
Part (a) is an immediate consequence of Lemma \ref{lemma:aux}(b1), the fact that the pencil $\mathcal{L}(\lambda)$ is a strong minimal bases pencil with $N_1(\lambda)=N_2(\lambda)=\Lambda_s(\lambda)\otimes I_s$, and the fact that the $(s+1)$th block of $\Lambda_s(\lambda)\otimes I_s$ is the identity matrix $I_n$.
Part (b) follows immediately from  Lemma \ref{lemma:aux}(b2) combined with the fact $\deg(\Lambda_s(\lambda)\otimes I_s)=s$.
\end{proof}
\end{theorem}

\subsection{Recovery procedures for matrix polynomials with even degree}\label{sec:recovery_even}

In Theorems \ref{thm:recover_eig_even1} and \ref{thm:recover_eig_even2} we show how to recover the eigenvectors associated with finite and infinite eigenvalues of an even-degree regular matrix polynomial with nonsingular leading or trailing coefficient  from any of its strong linearizations obtained from modified  block Kronecker pencils.

Modified  block Kronecker pencils are minimal bases pencils but not strong minimal bases pencils.
This means that we can not use the results in Lemma \ref{lemma:aux}, as we have done in the previous section.
For this reason, we start with Lemmas \ref{lemma:recover_aux1} and \ref{lemma:recover_aux2}, where the following matrix polynomial
\begin{equation}
\label{eq:Nhat}
\widehat{N}_k(\lambda):=
\begin{bmatrix}
0 & 1 & & & & 0 \\
\vdots & \lambda & 1 & & & \vdots \\
\vdots & \vdots & \ddots & \ddots & & \vdots \\
0 & \lambda^{t-2} & \cdots & \lambda & 1 & 0 \\
\end{bmatrix}\otimes I_n \in\FF[\lambda]^{(t-1)n\times (t+1)n}
\end{equation}
plays an important role.

Lemma \ref{lemma:recover_aux1} shows that modified  block Kronecker pencils  admit simple one-sided factorizations.
\begin{lemma}\label{lemma:recover_aux1}
Let $P(\lambda)=\sum_{k=0}^d P_k\lambda^k \in\FF[\lambda]^{n\times n}$ be an even-degree matrix polynomial. Let $\mathcal{L}(\lambda)$ be a modified block Kronecker pencil as in \eqref{eq:linearization_even2} satisfying \eqref{eq:first_condition} and \eqref{eq:second_condition}. Let $\Lambda_k(\lambda)$ and $\widehat{N}_k(\lambda)$ be the matrix polynomials in \eqref{eq:Lambda} and \eqref{eq:Nhat}, respectively.
Then,
\begin{equation}\label{eq:right_fact}
\mathcal{L}(\lambda)
\begin{bmatrix}
\Lambda_t(\lambda)\otimes I_n \\
\widehat{N}_t(\lambda)(\lambda B+A)(\Lambda_t(\lambda)\otimes I_n)
\end{bmatrix} =
e_{t+1}\otimes P(\lambda),
\end{equation}
where $e_\ell$ denotes the $\ell$th column of the identity matrix $I_d$.
\end{lemma}
\begin{proof}
Using \eqref{eq:first_condition} and \eqref{eq:second_condition}, the equality \eqref{eq:right_fact} follows from  a direct matrix multiplication.
\end{proof}

Lemma \ref{lemma:recover_aux2} relates any right eigenvector of a linearization obtained from a modified  block Kronecker pencil for an even-degree matrix polynomial $P(\lambda)$ with a right eigenvector of $P(\lambda)$.
\begin{lemma}\label{lemma:recover_aux2}
Let $P(\lambda)=\sum_{k=0}^d P_k\lambda^k \in\FF[\lambda]^{n\times n}$ be a regular even-degree matrix polynomial with nonsingular leading coefficient. Let $\mathcal{L}(\lambda)$ be a strong linearization of $P(\lambda)$ obtained from a modified block Kronecker pencil \eqref{eq:linearization_even} with $t=d/2$.
Then, any right eigenvector $z$ of $\mathcal{L}(\lambda)$ with finite eigenvalue $\lambda_0$ has the form
\[
\begin{bmatrix}
\Lambda_t(\lambda_0)\otimes I_n \\
\widehat{N}_t(\lambda_0)(\lambda_0 B+A)(\Lambda_t(\lambda_0)\otimes I_n)
\end{bmatrix}x
\]
for some right eigenvector $x$ of $P(\lambda)$ with finite eigenvalue $\lambda_0$. 
\end{lemma}
\begin{proof}
Let us denote by $\mathcal{N}_r(\mathcal{L}(\lambda_0))$ and $\mathcal{N}_r(P(\lambda_0))$, respectively, the right null space of the matrices $\mathcal{L}(\lambda_0)$ and  $P(\lambda_0)$.
From Lemma \ref{lemma:recover_aux1} we get that $x\in\mathcal{N}_r(P(\lambda_0))$ if and only if
\[
\begin{bmatrix}
\Lambda_t(\lambda_0)\otimes I_n \\
\widehat{N}_t(\lambda_0)(\lambda_0 B+A)(\Lambda_t(\lambda_0)\otimes I_n)
\end{bmatrix}x\in\mathcal{N}_r(\mathcal{L}(\lambda_0)).
\]
Let $\{x_1,\hdots,x_p\}$ be a linear independent basis of $\mathcal{N}_r(P(\lambda_0))$, and let us set
\[
z_i:=\begin{bmatrix}
\Lambda_t(\lambda_0)\otimes I_n \\
\widehat{N}_t(\lambda_0)(\lambda_0 B+A)(\Lambda_t(\lambda_0)\otimes I_n)
\end{bmatrix}x_i\in\mathcal{N}_r(\mathcal{L}(\lambda_0)).
\]
We claim that $\{z_1,\hdots,z_p\}$ is a linear independent basis of $\mathcal{N}_r(\mathcal{L}(\lambda_0))$.
Indeed, let $c_1,\hdots,c_p$ be constants such that $c_1z_1+\cdots+c_p z_p=0$. 
From the $(t+1)$th block entry of the previous equation we obtain $c_1x_1+\cdots+c_p x_p=0$ which implies $c_1=\cdots=c_p=0$.
Thus, the set $\{z_1,\hdots,z_p\}$ is a linearly independent set such that ${\rm span}\,\{z_1,\hdots,z_p\}\subseteq \mathcal{N}_r(\mathcal{L}(\lambda_0))$.
But notice that $\dim(\mathcal{N}_r(P(\lambda_0))) = \dim(\mathcal{N}_r(\mathcal{L}(\lambda_0)))$, since $\mathcal{L}(\lambda)$ is a linearization of $P(\lambda)$.
Therefore, $\{z_1,\hdots,z_p\}$ is a linear independent basis of $\mathcal{N}_r(\mathcal{L}(\lambda_0))$.
Finally, let $0\neq z\in \mathcal{N}_r(\mathcal{L}(\lambda_0))$.
By the previous considerations we get
\begin{align*}
z = &\sum_{i=1}^p \alpha_i z_i = 
\sum_{i=1}^p \alpha_i \begin{bmatrix}
\Lambda_t(\lambda_0)\otimes I_n \\
\widehat{N}_t(\lambda_0)(\lambda_0 B+A)(\Lambda_t(\lambda_0)\otimes I_n)
\end{bmatrix}x_i = \\
&\begin{bmatrix}
\Lambda_t(\lambda_0)\otimes I_n \\
-\widehat{N}_t(\lambda_0)(\lambda_0 B+A)(\Lambda_t(\lambda_0)\otimes I_n)
\end{bmatrix}\sum_{i=1}^p \alpha_i c_i = \\
&\begin{bmatrix}
\Lambda_t(\lambda_0)\otimes I_n \\
\widehat{N}_t(\lambda_0)(\lambda_0 B+A)(\Lambda_t(\lambda_0)\otimes I_n)
\end{bmatrix}x,
\end{align*}
where $0\neq x:=\sum_{i=1}^p \alpha_i x_i\in\mathcal{N}_r(P(\lambda_0))$.
\end{proof}

In Theorem \ref{thm:recover_eig_even1} we present the eigenvector recovery procedures for symmetric or skew-symmetric even-degree matrix polynomials with nonsingular leading coefficients. 
Recall that matrix polynomials with nonsingular leading coefficients do not have eigenvalues at infinity.
\begin{theorem}\label{thm:recover_eig_even1}
Let $P(\lambda)=\sum_{k=0}^d P_k\lambda^k \in\FF[\lambda]^{n\times n}$ be a regular even-degree matrix polynomial with nonsingular leading coefficient. Let $\mathcal{L}(\lambda)$ be a strong linearization of $P(\lambda)$ obtained from a modified  block Kronecker pencil \eqref{eq:linearization_even} with $t=d/2$.
Then, if $z\in\mathbb{F}^{nd\times 1}$ is a right eigenvector of $\mathcal{L}(\lambda)$ with finite eigenvalue $\lambda_0$, then the $(t+1)$th block of $z$ is a right eigenvector of $P(\lambda)$ with finite eigenvalue $\lambda_0$.
\end{theorem}
\begin{proof}
The result follows immediately from Lemma \ref{lemma:recover_aux2} combined with the fact that the $(t+1)$ block of $\Lambda_t(\lambda)\otimes I_n$ is the identity matrix $I_n$.
\end{proof}

When the matrix polynomial $P(\lambda)$ has nonsingular trailing coefficient, a  strong linearization $\mathcal{L}(\lambda)$ may be obtained from the reversal of a modified block Kronecker pencil \eqref{eq:linearization_even} of the matrix polynomial $\rev P(\lambda)$, as we have seen at the end of Section \ref{sec:structured-lin:even}.
In Theorem \ref{thm:recover_eig_even2} we present the eigenvector recovery procedures for this kind of strong linearizations.
Recall that matrix polynomials with nonsingular trailing coefficients do not have  the eigenvalue $\lambda_0=0$.
\begin{theorem}\label{thm:recover_eig_even2}
Let $P(\lambda)=\sum_{k=0}^d P_k\lambda^k \in\FF[\lambda]^{n\times n}$ be a regular even-degree matrix polynomial with nonsingular trailing coefficient, let $\widetilde{\mathcal{L}}(\lambda)$ be a strong linearization of $\rev P(\lambda)$ obtained from  a modified block Kronecker pencil \eqref{eq:linearization_even} with $t=d/2$, and let $\mathcal{L}(\lambda)=\rev \widetilde{\mathcal{L}}(\lambda)$, which is a strong linearization of $P(\lambda)$.
Then the following statements hold.
\begin{enumerate}
\item[{\rm (a)}] If $z\in\mathbb{F}^{nd\times 1}$ is a right eigenvector of $\mathcal{L}(\lambda)$ with nonzero finite eigenvalue $\lambda_0$, then the $(t+1)$th block of $z$ is a right eigenvector of $P(\lambda)$ with finite eigenvalue $\lambda_0$.
\item[{\rm (b)}] If $z\in\mathbb{F}^{nd\times 1}$ is a right eigenvector of $\mathcal{L}(\lambda)$ for the eigenvalue $\infty$, then the $(t+1)$th block of $z$ is a right eigenvector of $P(\lambda)$ for the eigenvalue $\infty$.
\end{enumerate}
\end{theorem}
\begin{proof}
Using Lemma \ref{lemma:rev}, parts (a) and (b) follow from the following argument.
Let $z$ be a right eigenvector of $\mathcal{L}(\lambda)$ with eigenvalue $\lambda_0\neq 0$, and let us denote by $x$ the $(t+1)$th block of the vector $z$.
Then, $z$ is a right eigenvector of $\rev \mathcal{L}(\lambda)$ with eigenvalue $1/\lambda_0$.
From Theorem \ref{thm:recover_eig_even1} we get that $x$ is a right eigenvector of $\rev P(\lambda)$ with eigenvalue $1/\lambda_0$, which implies that $x$ is a right eigenvector of $P(\lambda)$ with eigenvalue $\lambda_0$.
\end{proof}

\section{Conclusions}\label{sec:conclusions}

In this paper we have introduced a new framework for symmetric and skew-symmetric linearizations that might include most of the symmetric and skew-symmetric linearizations obtained from Fiedler pencils with repetitions \cite{FPR1}.
To this aim we have introduced the families of (modified) symmetrizable block Kronecker pencils and (modified) skew-symmetrizable block Kronecker pencils, which belong, respectively, to the sets of  minimal bases pencils and strong minimal bases pencils \cite{Fiedler-like}.
Symmetrizable and skew-symmetrizable block Kronecker pencils have been used to construct structure-preserving linearizations of symmetric and skew-symmetric odd-degree matrix polynomials, and we have shown that these linearizations are  strong regardless of whether the matrix polynomials are regular or singular.
Among them, the simplest one is 
\[
\left[ \begin{array}{cccc|ccc}
\lambda P_d + P_{d-1} & & & & -I_n \\
& \lambda P_{d-2}+P_{d-3} & & & \lambda I_n & \ddots\\
& & \ddots & & & \ddots & -I_n \\
& & & \lambda P_1+P_0 &  &  & \lambda I_n \\ \hline
-\sigma I_n & \sigma \lambda  I_n & & & &  \\
& \ddots & \ddots & & & & \\
& & -\sigma I_n & \lambda \sigma I_n
\end{array} \right],
\]
where $\sigma=1$ if $P(\lambda)$ is symmetric or $\sigma=-1$ if $P(\lambda)$ is skew-symmetric, which are a permuted version of the famous block-tridiagonal symmetric or skew-symmetric companion forms \cite{Greeks,Greeks2,Skew}.
Even-degree structured matrix polynomials do not always have structure-preserving linearizations.
However, we have shown that  modified symmetrizable and skew-symmetrizable block Kronecker pencils can be used to construct structure-preserving strong linearizations for symmetric or skew-symmetric even-degree matrix polynomials when their leading or trailing coefficients are nonsingular.
Among these linearizations (when the leading coefficients are nonsingular), the simplest one is
\[
\left[\begin{array}{ccccc|ccc}
-P_d & \lambda P_d & & & & 0 & \cdots & 0 \\
\lambda P_d & \lambda P_{d-1}+P_{d-2} &  &  &  & -I_n \\
& & \lambda P_{d-3}+P_{d-4} & & & \lambda I_n & \ddots  \\
& & & \ddots & & & \ddots & -I_n \\
& & & & \lambda P_1+P_0 & & & \lambda I_n \\ \hline
0 & -\sigma I_n & \sigma \lambda I_n & & & & \\
\vdots & & \ddots & \ddots & & & \\
0 & & & -\sigma I_n & \sigma \lambda I_n & & 
\end{array}\right].
\]
where $\sigma=1$ if $P(\lambda)$ is symmetric or $\sigma=-1$ if $P(\lambda)$ is skew-symmetric, which can be permuted into a block-tridiagonal pencil.
Since the families of symmetric and skew-symmetric block Kronecker pencils and modified symmetric and skew-symmetric block Kronecker pencils belong, respectively, to the  sets of strong minimal bases pencils or minimal bases pencils, they inherit all their desirable properties for numerical applications.
In particular, we have shown that eigenvectors, minimal indices, and minimal bases
of matrix  polynomials are easily recovered from those of any of the linearizations constructed in this work.

\end{document}